\documentclass{amsart}
\usepackage[mathscr]{eucal}
\usepackage{tikz-cd}
\usepackage{graphicx}
\usepackage{amscd}
\usepackage{amsmath}
\usepackage{amsthm}
\usepackage{amsxtra}
\usepackage{calc} 
\usepackage{amsfonts}
\usepackage{amssymb}
\usepackage{pdfsync}
\usepackage{mdwlist}
\usepackage[all]{xy}
\usepackage[colorlinks, bookmarks=true]{hyperref}

\numberwithin{equation}{section}
\newtheorem{theorem}{Theorem}[section]
\theoremstyle{plain}

\newtheorem{corollary}[theorem]{Corollary}
\newtheorem{corollary-definition}[theorem]{Corollary-Definition}

\newtheorem{definition}[theorem]{Definition}

\newtheorem{lemma}[theorem]{Lemma}

\newtheorem{proposition}[theorem]{Proposition}

\numberwithin{equation}{section}

\theoremstyle{definition}
\newtheorem{example}[theorem]{Example}
\newtheorem{remark}[theorem]{Remark}

\newcommand{\Ker}{\mathrm{Ker}\,}

\newcommand{\Hom}{\mathrm{Hom}}
\newcommand{\Ext}{\mathrm{Ext}}

\newcommand{\Z}{\mathbb{Z}}
\newcommand{\Q}{\mathbb{Q}}

\newcommand{\Mod}{\mathrm{Mod}}
\newcommand{\Modst}{\underline{\mathrm{Mod}}}
\newcommand{\h}{\underline{h}}

\newcounter{hours}
\newcounter{minutes}

\begin{document}
\title{The stable category of a left hereditary ring}

\author{Alex Martsinkovsky}
\address{Mathematics Department\\
Northeastern University\\
Boston, MA 02115, USA}
\email{alexmart@neu.edu}
\author{Dali Zangurashvili}
\address{Andrea Razmadze Mathematical Institute\\
Tbilisi Centre for Mathematical Sciences, 6 Tamarashvili St.\\
Tbilisi, 0177, Georgia}
\email{dalizan@rmi.ge}
\thanks{The first author was supported by the Collaborative Research Centre 701 "Spectral Structures and Topological Methods in Mathematics" at the University of Bielefeld during his visit in May - July of 2014. He thanks the University of Bielefeld for providing ideal conditions for work.}
\thanks{A major part of this paper was prepared during the second author's visit to Northeastern University (USA) in October-November 2012 under the financial support from the Short-Term Individual Travel Grant from Shota Rustaveli National Science Foundation (Ref. 03/109), which she gratefully acknowledges. She also gratefully acknowledges the Research Grant  DI/ 18/5-113/13 from the same foundation and the  Volkswagen Foundation's Research Grant (Ref. 85 989).}

\date{\today, \setcounter{hours}{\time/60} \setcounter{minutes}{\time-\value
{hours}*60} \thehours\,h\ \theminutes\,min}

\begin{abstract}
The (co)completeness problem for the (projectively) stable module category of an associative ring is studied. (Normal) monomorphisms and (normal) epimorphisms in such a category are characterized. As an application, we give a criterion for the stable category of a left hereditary ring to be abelian. By a structure theorem of Colby-Rutter, this leads to an explicit description of all such rings.
\end{abstract}
\maketitle
\tableofcontents
\section{Introduction}

Throughout the paper, a ring will always be an associative ring with identity, and a module a left module. Given a ring, the corresponding stable module category has modules as objects, while its morphisms are equivalence classes of module homomorphisms modulo homomorphisms factoring through projectives. This category was introduced by Eckmann and Hilton in the 1950s, with a goal to build an algebraic ``prototype'' for duality in homotopy theory~(\cite[Ch. 13]{H}). Soon it found uses in modular representation theory of finite groups and, eventually, in representation theory of artin algebras.

One may think that the stable category is not abelian (and this is sometimes claimed in the folklore), but if $\Lambda$ is semisimple, then the corresponding stable category consists of zero objects only and is thus trivially abelian. There now arises a natural question of whether there is a nontrivial example of an abelian stable category. Notice that if we are allowed to replace the category of all modules by a subcategory, then it is not difficult  to come up with such an example. It is easy to see that the trivial group is the only projective in the abelian category of finite abelian groups, and therefore the corresponding stable category is isomorphic to the original category.
%
%
Thus, we keep the entire module category in the statement of the problem and assume that the ring $\Lambda$ is not semisimple. 

Our first nontrivial example of an abelian stable category comes from representations of quivers, more precisely, those of $A_{2}$:
\[
\xymatrix
	{
	1 \ar[r] & 2
	}
\]
Let $k$ be a field. The path algebra $kA_{2}$ is isomorphic to the algebra of triangular $2 \times 2$ matrices with entries in $k$.  Let $P_{i}$ and $ S_{i}$ denote, respectively, the projective and the simple module corresponding to the vertex $i = 1,2$. It is well-known that $kA_{2}$ is of finite representation type, with only indecomposables being $P_{2} = S_{2}$, $P_{1}$, and $S_{1}$. Passing to the quotient modulo projectives, we have that, up to isomorphism,  $S_{1}$ is the only nonzero indecomposable module in the stable category of \texttt{finitely generated} $kA_{2}$-modules, making that category equivalent to the category of finite-dimensional $k$-vector spaces, and hence abelian. To describe the stable category of the entire category of $kA_{2}$-modules, we recall (\cite[Corollary 4.4]{RT}) that, for algebras of finite representation type, each module, finitely generated or not, is a direct sum of indecomposables. It now follows that the stable category of all $kA_{2}$-modules is equivalent to the abelian category of all $k$-vector spaces.

The above example fits  a more general pattern. Let $A_{n}$, 
$n \geq 2$ be the equioriented quiver
\[
\xymatrix
	{
	1 \ar[r] & 2 \ar[r] & 3 \ar[r] & \ldots \ar[r] & n
	}
\] 
By analyzing the Auslander-Reiten quiver of $kA_{n}$, one can see that the stable category of finitely generated $kA_{n}$-modules is equivalent to the category of finitely generated modules over $kA_{n-1}$, and is thus abelian. Since $kA_{n}$ is also of finite representation type, we can use~\cite[Corollary 4.4]{RT} again and deduce that the stable category of all $kA_{n}$-modules is equivalent to the abelian category of all $kA_{n-1}$-modules.

With the motivating examples above, our goal in this paper is to 
determine when the stable category is abelian. In the case of a left hereditary ring, we give a complete answer.  Our main result, Theorem~\ref{main}, says that the stable category of such a ring is abelian if and only if the injective envelope of the ring viewed as a left module over itself is projective. By a structure theorem of Colby-Rutter, these are precisely finite direct products of complete blocked triangular matrix algebras over division rings.

The paper is organized as follows. In Section~\ref{Notation}, we set up notation and recall  basic facts about stable categories. 

In Section~\ref{Quotient}, we deal with the quotient functor from modules to modules modulo projectives. In particular, we characterize rings over which the quotient functor has left or right adjoints. 

In Section~\ref{StableMonos}, we characterize monomorphisms in stable categories. The case of a left hereditary ring is considered in more detail, which leads to seven new characterizations of such rings. As a consequence, we have that the stable category of a left hereditary ring is finitely complete. 

Section~\ref{StableEpis} deals with epimorphisms in stable categories. Theorem~\ref{char-epi} characterizes epimorphisms in stable categories in terms of null-homotopy of chain maps associated with certain pushouts. For left hereditary rings, we provide yet another criterion for a homomorphism to give rise to an  epimorphism in the stable category. For left hereditary rings, Theorem~\ref{iso}  gives a new necessary and sufficient condition for a homomorphism to represent an isomorphism in the stable category. Unlike Heller's general criterion, our condition is formulated in terms of submodules, rather than overmodules.  

In Section~\ref{epis-and-torsion}, we show that the study of  epimorphisms in the stable category reduces, to a large extent, ``modulo torsion''. Also, at the end of the section, we show that the full subcategory determined by the torsionfree modules is reflective in the stable category of the ring.

In Section~\ref{NormalMonos}, we study normal monomorphisms in the stable category. In particular, we show (Theorem~\ref{normal}) that all monomorphisms in the stable category of a left hereditary ring are normal if and only if the injective envelope of the ring, viewed as a left module over itself, is projective. As a consequence, the stable category of such a ring is well-powered. 

Section~\ref{NormalEpis} deals with normal epimorphisms. Compared with all epimorphisms, they give rise to null-homotopies of chain maps associated with certain additional pushouts (Proposition~\ref{pushout-split}). Over left hereditary rings, the existence of such null-homotopies  implies the normality (Theorem~\ref{normal-epi}). Lemma~\ref{conormal} shows that if the injective envelope of a left hereditary ring is projective, then the corresponding stable category is conormal. To determine when the converse is true, we show (Theorem~\ref{existence}) that when a left hereditary ring has the DCC on direct summands of itself and has a non-projective injective envelope, then there exists a nonzero projective module with a stable injective envelope; this construction gives rise to a bimorphism in the stable category which is not an isomorphism. At the end of the section, we briefly mention factorization systems, and give a necessary condition for the stable category of a left hereditary ring to admit epi-mono factorizations. 

In the last Section~\ref{Cokernels}, we show that if the injective envelope of a left hereditary ring is projective, then the corresponding stable category is cocomplete. This leads to  Theorem~\ref{main}, the main result of the paper. 


Since this paper straddles the area between module theory and category theory, a substantial effort has been made to present the results in a way accessible to a wide audience. Yet, if the reader needs more background from category theory, we recommend~\cite{AHS}, \cite{B}, and~\cite{M}. For a concise and focused treatment of ring- and module-theoretic concepts, the reader is referred to~\cite{AF}.

The authors thank Kiyoshi Igusa for a helpful comment on the above examples. 
Special thanks go to Oana Veliche who, carefully read the initial drafts of the paper and helped improve its readability. The second author expresses her sincere gratitude to Alex Martsinkovsky and Oana Veliche for their extraordinary kindness and hospitality during her visit to Northeastern University. 

\section{Notation and preliminaries}\label{Notation}

\textbf{Blanket assumptions}. As we mentioned in the introduction, throughout this paper, a ring means an associative ring with identity, and a module is a left module. The symbol 
$\mathrm{ker} f$ (respectively, $\mathrm{coker} f$) will denote the kernel (respectively, the cokernel) of the morphism $f$. The symbol $\mathrm{Ker} f$ (respectively, 
$\mathrm{Coker} f$) will denote the domain (respectively, the codomain) of the kernel (respectively, of the cokernel) of $f$. Also, a (co)limit in a category will always mean a small (co)limit. 
\medskip

Given a ring $\Lambda$, the category of left $\Lambda$-modules is denoted by  $\Lambda$-$\Mod$, and the corresponding stable category is denoted by $\Lambda$-$\Modst$, with the set of morphisms customarily denoted by $\underline{\Hom}(A,B)$, or simply $(\underline{A,B})$,  for any $\Lambda$-modules $A$ and $B$.  We shall say that homomorphisms $f$ and $g$ with the same domains and codomains are equivalent if $f-g$ factors through a projective. The morphism in the stable category represented by $f$ will be denoted by $\underline{f}$. Thus, 
$\underline{f} = 0$ in the stable category if and only if $f$ factors through a projective. On a rare occasion, we shall write $\underline{A}$ to indicate that the module $A$ is viewed as an object of the stable category.

The stable category is additive, with the biproduct induced by that in the original module category. The quotient functor 
 \[
 \mathcal{Q} : \Lambda\text{-}\mathrm{Mod} \to \Lambda\text{-}\underline{\mathrm{Mod}},
 \]
 defined tautologically on objects and by the quotient map on  morphisms, is clearly additive, full, and dense. It  has the following universal property. If~$\mathbf{A}$ is an additive category and $\mathscr{F} : \Lambda\text{-}\mathrm{Mod} \to \mathbf{A}$ is an additive functor vanishing on projectives, then there is a unique (up to isomorphism) additive functor 
 $\mathscr{G} : \Lambda\text{-}\underline{\mathrm{Mod}} \to \mathbf{A}$ making the diagram
\[
\xymatrix
	{
	\Lambda\text{-}\mathrm{Mod} \ar[r]^{\mathcal{Q}} 
	\ar[d]_{\mathcal{F}}
	& \Lambda\text{-}\underline{\mathrm{Mod}} 				\ar@{.>}[dl]^{\mathcal{G}}
\\
	\mathbf{A}
	&
	}
\] 
commute (up to isomorphism).
\smallskip

For any $\Lambda$-module $X$, we set $h^{\underline{X}} := (\underline{X,-})$ and $h_{\underline{X}} := (\underline{-, X})$, and view these as additive functors from the stable category of $\Lambda$-modules to the category of abelian groups. For a homomorphism $f : X \to Y$ of $\Lambda$-modules, we set $h^{\underline{f}} := (\underline{f,-})$ and $h_{\underline{f}} := (\underline{-, f})$. Composing these functors and natural transformations with~$\mathcal{Q}$, we define $\h^{X} := h^{\underline{X}}\mathcal{Q}$, 
$\h_{X} := h_{\underline{X}}\mathcal{Q}$, $\h^{f} := h^{f}\mathcal{Q}$, and 
$\h_{f} := h_{f}\mathcal{Q}$. These are additive functors from 
$\Lambda$-modules to abelian groups and, respectively, natural transformations between such functors. To simplify the notation, we drop the functor $\mathcal{Q}$ from the above definitions and, in an abused notation, simply write:

\begin{itemize}
\item[(i)] $\h^{X} = h^{\underline{X}}$ and $\h_{X} = h_{\underline{X}}$;
\smallskip
\label{underline}
\item[(ii)] $\h^{f}= h^{\underline{f}}$ and $\h_{f}= h_{\underline{f}}$,
\end{itemize}
\smallskip

To check whether a module is a zero object in the stable category, we have the following well-known and easily verified criteria.

\begin{lemma}\label{criteria}
 Let $A$ be a $\Lambda$-module. The following are equivalent:
 
\begin{enumerate}
  \item $A$ is a zero object in the stable category;
  \smallskip
  \item $A$ is projective;
  \smallskip
  \item $\h^{A} = 0$;
  \smallskip
  \item $\h_{A} = 0$. \qed
\end{enumerate} 
\end{lemma}

Another well-known and useful fact is given by

\begin{lemma}
The isomorphism class of any morphism in the stable category contains a morphism that can be represented by an epimorphism in $\Lambda$-$\Mod$.
\end{lemma}

\begin{proof}
 Let $f : M \to N$ be a homomorphism of $\Lambda$-modules. If 
$g : P \to N$ is an epimorphism with $P$ projective, then 
$f \bot g : M \oplus P \to N$ is the desired homomorphism.
\end{proof}

To test whether a homomorphism $f$ gives rise to an isomorphism in the stable category, we have the following criterion of Heller~\cite[Theorem 2.2]{H2}.

\begin{theorem}\label{heller}
Let $f : A \to B$ be a homomorphism of $\Lambda$-modules. Then  $\underline{f}$ is an isomorphism if and only if there are projective modules $P$ and $Q$ and an isomorphism $\tilde{f}:A\oplus P \rightarrow B\oplus Q$ making the following diagram

\begin{equation}
\begin{gathered}
 \xymatrix{A\oplus P\ar[r]^{\tilde{f}}&B\oplus Q\ar[d]^{p}\\
A\ar[u]^{i}\ar[r]^{f}&B} 
\end{gathered} 
\end{equation}
where $i$ and $p$ are, respectively, the canonical inclusion and the canonical projection, commute. \qed
\end{theorem}

For our purposes, it is convenient to introduce the following

\begin{definition}
 A module is \texttt{totally stable} if any endomorphism of that module is an automorphism whenever it is an automorphism modulo projectives.\footnote
 	{
This class of modules, without a name, was already considered in \cite{Aus69}, \S 2.
	} 
\end{definition}
It is not difficult to see~(\cite[Lemma 2.1, c)]{Aus69}) that a module $A$ is totally stable if and only if any endomorphism of $A$ factoring through a projective is in the radical of $\mathrm{End}\,A$. 


\begin{lemma}\label{zero dual}
 Equivalent homomorphisms with a zero dual domain are equal. Any module with a zero dual is totally stable.
\end{lemma}
\begin{proof}
Indeed, any homomorphism from such a module factoring through a projective must be zero.
\end{proof}

Recall that a module is said to be \texttt{stable} if it has no nonzero projective summands. If no left ideal of $\Lambda$ not contained in the radical of $\Lambda$ is stable, then a finitely generated module is totally stable if and only if it is stable~\cite[Proposition 2.5 and the following note]{Aus69}. A ring satisfies the above condition whenever it is semiperfect or left hereditary, [ibid.]. In fact, as Lemma~\ref{hered-stable} below shows, over a left hereditary ring, any module, finitely generated or not, is totally stable if and only if it is stable.

\begin{lemma}\label{hered-stable}
 Let $\Lambda$ be a ring and $A$ a $\Lambda$-module. Consider the following conditions:
  
\begin{enumerate}
 \item $A^{\ast} = 0$, where $(-)^{\ast} = \Hom (-,\Lambda)$;
 \smallskip
 \item $A$ is totally stable;
 \smallskip
 \item $A$ is stable.
\end{enumerate}
Then (1) $\Rightarrow$ (2) $\Rightarrow$ (3). If $\Lambda$ is left hereditary, then these conditions are equivalent.
\end{lemma}

\begin{proof}
 (1) $\Rightarrow$ (2). This is Lemma~\ref{zero dual}.
 \smallskip
 
 (2) $\Rightarrow$ (3). Suppose $A \simeq B \oplus P$, where $P$ is projective. Then the direct sum of the identity of $B$ and the zero endomorphism of $P$ is an automorphism of $A$ because it gives rise to an automorphism in the stable category. It follows that $P = (0)$. 
 \smallskip
 
 Now assume that $\Lambda$ is left hereditary.
 \smallskip
 
 (3) $\Rightarrow$ (1). The image of any linear form on $A$, being an ideal, is projective, and is therefore a projective summand of $A$. Since $A$ is stable, that image is zero.
\end{proof}

Combining Lemmas~\ref{zero dual} and~\ref{hered-stable}, we have

\begin{lemma}\label{stabledomain}
Over a left hereditary ring, equivalent homomorphisms with a stable domain are equal. \qed
\end {lemma}

Returning to the last two conditions of Lemma~\ref{criteria}, we would like to relate, in subsequent sections, the functors of type 
$\h^{A}$ and $\h_{A}$ by various exact sequences. A proper context to speak of exactness is a category with kernels and cokernels, with a natural candidate here being the category of functors between two fixed categories as objects and natural transformations between those functors as morphisms. However, if the domain category is not small, the morphisms between two objects need not form a set. One way around this obstacle is to use the notion of quasicategories~\cite[3.49 - 3.51]{AHS}. However, for our purposes it suffices to consider only componentwise exactness, whence the following definition.

\begin{definition}
 Let $F$, $G$, and $H$ be functors  $\Lambda$-$\mathrm{Mod} \to \mathrm{Ab}$ to the category of abelian groups. A sequence  
 $0 \to F \to G \to H \to 0$ of natural transformations is said to be 
 c-exact if it is exact in each component.  Similarly, define c-kernels, c-cokernels,  c-monomorphisms and c-epimorphisms.
\end{definition}

\section{The quotient functor $\mathcal{Q}$}\label{Quotient}

In this section we give criteria for the existence of adjoints of the quotient functor $\mathcal{Q}: \Lambda$-$\mathrm{Mod} \to \Lambda$-$\underline{\mathrm{Mod}}$. To describe the class of rings for which the quotient functor has a \texttt{left} adjoint, we first need a preliminary result characterizing rings over 
which~$\mathcal{Q}$ preserves products.

\begin{proposition}\label{products}
Given a ring $\Lambda$, the following are equivalent:
\begin{enumerate}
 \item the quotient functor 
$\mathcal{Q} : \Lambda\text{-}\mathrm{Mod} \to \Lambda\text{-}\underline{\mathrm{Mod}}$ preserves  products;
\smallskip
\item the direct product of any  family of projectives is projective;
\smallskip
\item $\Lambda$ is left perfect and right coherent.
\end{enumerate}
\end{proposition}

\begin{proof}

(1) $\Rightarrow$ (2). 
Since $\mathcal{Q}$ preserves products, a product of any family of projectives exists in the stable category and is represented by the product of these modules in the original category. Since projectives are zero objects in the stable category, their product is also zero, i.e., the product of projectives is projective in the original category. 
\smallskip

(2) $\Rightarrow$ (1). 
Let $(A_{i})_{i\in I}$ be a family of modules indexed by a set  $I$. We want to show that the product of the $A_{i}$ in the stable category exists and is represented by the product $(\prod_{i \in I} A_{i}, p_{i})$ in the original category. Given a module~$B$ and a family of homomorphisms $(f_i : B \to A_i)_{i\in I}$,  we have a homomorphism $f:B \to \prod_{i \in I} A_{i}$ such that $p_i f=f_{i}$ for any $i\in I$. Applying $\mathcal{Q}$, we have the desired commutation relations for $\underline{f}$ in the stable category. To show that $\underline{f}$ is unique, assume that 
$\underline{p_i}\underline{g} = \underline{f_i}$ for some 
$g:B \to \prod_{i \in I} A_i$ and for all~$i$, i.e., each $p_i (f - g)$ factors through a projective~$P_i$. Then $f - g$ factors through the product of the $P_i$, which is  projective by the assumption. Thus $\underline{f}=\underline{g}$. 
 \smallskip
 
(2) $\Leftrightarrow$ (3). This is a result of Chase~\cite[Theorem 3.3]{Chase}.
\end{proof}

\begin{corollary}\label{small-products}
 If $\Lambda$ is left perfect and right coherent, then
 $\Lambda$-$\Modst$ has  products. \qed
\end{corollary}

Specializing Proposition~\ref{products} to the case of a left hereditary ring, we can now answer the question of when 
$\mathcal{Q}$ has a left adjoint.

\begin{theorem}\label{left-adjoint}
 The following are equivalent: 
 
\begin{enumerate}
 \item The ring $\Lambda$ is left hereditary and satisfies the equivalent conditions of Proposition~\ref{products};
 \smallskip
 \item The quotient functor $\mathcal{Q}$ preserves  limits;
  \smallskip
 \item The quotient functor $\mathcal{Q}$ has a left adjoint.
\end{enumerate} 
\end{theorem}
 
\begin{proof}
 (1) $\Rightarrow$ (2). By Theorem~\ref{characterization}, proved in the next section, $\mathcal{Q}$ preserves kernels and, since it is additive, it preserves equalizers. As $\mathcal{Q}$ preserves  products, it preserves limits. 
  \smallskip

(2) $\Rightarrow$ (3). Since $\Lambda$-$\mathrm{Mod}$ has a cogenerator, $\mathcal{Q}$ has a left adjoint by virtue of the special adjoint functor theorem~\cite[Corollary to Theorem~5.8.2]{M}.
 \smallskip

(3) $\Rightarrow$ (1). If $\mathcal{Q}$ admits a left adjoint, it preserves kernels and  products. Using 
Theorem~\ref{characterization} again, we have that $\Lambda$ is left hereditary.
\end{proof}

\begin{corollary}\label{hered-complete}
 If $\Lambda$ is left hereditary, left perfect, and right coherent, then the stable category of $\Lambda$ is complete. \qed
\end{corollary} 


Similar to Proposition~\ref{products}, one easily proves

\begin{proposition}\label{P: quot-pres-coprod}
For any ring, the quotient functor preserves  coproducts, and hence the stable category of any ring has  coproducts. \qed
\end{proposition}

Now we can characterize rings over which $\mathscr{Q}$ has a right adjoint.

\begin{proposition}

The following are equivalent:

\begin{enumerate}

 \item $\mathcal{Q}$ has a right adjoint;
 \smallskip
 
 \item $\mathcal{Q}$ preserves cokernels;
  \smallskip
  
  \item $\mathcal{Q}$ preservers colimits;
  \smallskip 
  
 \item $\mathcal{Q}$ preservers finite colimits;
  \smallskip
  
 \item $\mathcal{Q}$ preserves epimorphisms;
  \smallskip
  
 \item $\Lambda$ is semisimple.
  \smallskip
  
\end{enumerate}

\end{proposition}

\begin{proof}

(1) $\Rightarrow$ (2). Being left adjoint, $\mathscr{Q}$ preserves coequalizers, and since it preserves the zero morphisms, it preserves cokernels.
 \smallskip

(2) $\Rightarrow$ (3). Since the coequalizer of two module homomorphisms is the cokernel of their difference, $\mathcal{Q}$ preserves coequalizers. By Proposition~\ref{P: quot-pres-coprod}, $\mathcal{Q}$ also preserves  coproducts, and therefore it preserves colimits.
 \smallskip

(3) $\Rightarrow$ (4). Trivial.
 \smallskip

(4) $\Rightarrow$ (5). This is well-known. 
 \smallskip

(5) $\Rightarrow$ (6) Given an arbitrary $\Lambda$-module $L$,
there is an epimorphism $f : P \to L$ with a projective $P$. Since 
$\underline{f}$ is also an epimorphism, $L$ has to be projective and, therefore, $\Lambda$ is semisimple. 
 \smallskip

(6) $\Rightarrow$ (1). If $\Lambda$ is semisimple, then all objects of the stable category are zero, hence $\mathcal{Q}$ has a right adjoint, which sends a zero object to zero. 
\end{proof}

\section{Monomorphisms in the stable category}\label{StableMonos}

We begin with a simple observation, which is valid over an arbitrary ring and will be used several times. 
\begin{lemma}\label{split}
 Let $0 \to A \overset{f}{\to} B \overset{g}\to C \to 0$ be a split exact sequence. Then:
 
\begin{enumerate}
 \item[(a)] $\underline{g}$ is an isomorphism in the stable category if and only if $A$ is projective.
 \smallskip
 \item[(b)] $\underline{f}$ is an isomorphism in the stable category if and only if $C$ is projective.
\end{enumerate}
\end{lemma}

\begin{proof}
(a) The above split exact sequence gives rise to a c-exact sequence of functors
\[
0 \to \h^{C} \overset{\h^{g}}{\to} \h^{B} \overset{\h^{f}}\to \h^{A} \to 0.
\]
By relations (i) and (ii) on p.~\pageref{underline} and Yoneda's lemma, $\underline{g}$ is an isomorphism if and only if $\h^{g}$ is. The latter is equivalent to $\h^{A} = 0$, i.e., to $A$ being projective.
\smallskip

(b) Similar to (a).
\end{proof}

Over an arbitrary ring, if a module homomorphism is a split monomorphism or a split epimorphism, then the same is obviously true for its class in the stable category. For a partial converse, we have

\begin{lemma}\label{split-mono}
Let $\Lambda$ be a ring and $f : A \to B$ a homomorphism of $\Lambda$-modules. If $A$ is totally stable then the following are equivalent: 
\begin{enumerate}
 \item $\underline{f}$ is a split monomorphism;
 \smallskip
 \item $f$ is a split monomorphism.
\end{enumerate}
\end{lemma}

\begin{proof}
Assume that $\underline f$ is a split monomorphism with splitting $\underline{g} : B \to A$. Then $\underline{gf} = 1_{A}$ and hence $gf$ is an isomorphism. The other implication is immediate.
\end{proof}

Another useful fact about morphisms in the stable category is given by 

\begin{proposition}\label{half-exact}
\begin{itemize}
 \item[(a)] For any module $X$, the functor $\h^{X}$ is half-exact.
 \smallskip
\item[(b)] For any short exact sequence $0 \to A \overset{g}{\to} B \overset{f}{\to} C \to 0$ in $\Lambda$-$\mathrm{Mod}$, the sequence
\[
\h_{A} \overset{\h_{g}}{\longrightarrow} \h_{B} \overset{\h_{f}}{\longrightarrow} \h_{C}
\] 
is c-exact.
\end{itemize}
  \end{proposition}
 
\begin{proof}
(a) Let $0 \to A \overset{g}{\to} B \overset{f}\to C \to 0$ be a short exact sequence. It is clear that $\h^{X}_{f} \h^{X}_{g} =0$, so we only need to show that whenever $\h^{X}_{f} (\underline{j}) =0$ for some $j : X \to B$, there is  $l : X \to A$ such that $\h^{X}_{g}(\underline{l}) = \underline{j}$. Thus we have a commutative diagram of solid arrows
\[
\xymatrix
	{
	&& X \ar@{..>}[ld]_{l} \ar[r] \ar[d]^{j} & Q \ar[d] \ar@{..>}[dl] & \\
	0 \ar[r]  & A  \ar[r]^{g} & B \ar[r]^{f} & C \ar[r] & 0 
	}
\] 
where $Q$ is a projective module. The dotted arrow from $Q$ to $B$ is a lifting of $Q \to C$. The difference between $j$ and the composition $X \to Q \to B$ composes with $f$ to zero, and thus lifts over $g$, producing a second dotted arrow, which is the desired homomorphism $l$.

(b) This is a restatement of (a) in terms of the contravariant Hom-functor.
\end{proof}

\begin{proposition}\label{mono}
 Let $\Lambda$ be an arbitrary ring and $f : A \to B$ a homomorphism of $\Lambda$-modules. 
 
\begin{itemize}
 \item[(a)] If $\underline{f}$ is a monomorphism in the stable category, then $\mathrm{ker}f$ factors through a projective.
 \smallskip
 \item[(b)] If $f$ is an epimorphism and $\mathrm{ker}f$ factors through a projective, then $\underline{f}$ is a monomorphism in the stable category.
\end{itemize}
\end{proposition}

\begin{proof}
(a) Since $f \ker f = 0$ and $\underline{f}$ is a monomorphism, 
$\ker f$ is zero in the stable category, and hence factors through a projective.
\smallskip

(b) We need to show that $h_{\underline{f}}$ is a c-monomorphism. But this follows from Proposition~\ref{half-exact} applied to the short exact sequence $0 \to \Ker f \overset{\ker f}{\to} A \overset{f}{\to} B \to 0$, and the fact that $h_{\underline{\ker f}} = 0$.
\end{proof}

The just proved result allows to construct an example of a simple module which has a nonzero proper subobject in the stable category.

\begin{example}
Let $R$ be a commutative Gorenstein local ring of Krull dimension 1 and of infinite global dimension with residue field $k$. (Take, for example, the local ring of a singular point on a plane algebraic curve.) Then the projective dimension of $k$ is infinite and 
$\Ext^{1}(k,R) \simeq k  \neq 0$. Let $0 \to R \to A \overset{p}{\to} k \to 0$ be a non-split extension. By Proposition~\ref{mono}, $\underline{p}$ is a monomorphism. Notice that $A$ is nonzero in the stable category, since otherwise $A$ would be projective, making $k$ a  module of finite projective dimension. We claim that, in the stable category, $A$ is a proper subobject of $k$, i.e., $\underline{p}$ is not an isomorphism. To see that, we first dualize the above short exact sequence into $k$, obtaining a long cohomology exact sequence
\[
\xymatrix
	{
	0 \ar[r]
	& (k, R) \ar[r]
	& (k, A) \ar[r]
	& (k, k) \ar[r]
	& \Ext^{1}(k, R) \ldots
	}
\] 
Under the last map, the identity on $k$ goes to the class of the above short exact sequence, which is non-split. Hence this map is nonzero and is therefore an isomorphism. Thus, the first map is also an isomorphism. Since the depth of $R$ is 1, we have $(k,R) = 0$ and therefore $(k,A) = 0$. Now suppose that $\underline{p}$ is an isomorphism. By Heller's Theorem~\ref{heller}, there would be free modules $P$ and 
$Q$ and an isomorphism $A \oplus P \simeq k \oplus Q$. But we just saw that 
$(k,A) = 0$. Furthermore, since $(k,R) = 0$, we also have $(k,P) = 0$. Thus $k$ cannot be a direct summand of the left-hand side, a contradiction. In summary, $A$ is a nonzero proper subobject of $k$.\footnote{The reader familiar with maximal Cohen-Macaulay approximations will immediately recognize in the above short exact sequence a minimal mCM approximation of $k$. This leads to an alternative argument, which goes as follows. Let $0 \to F \to A \overset{p}{\to} k \to 0$, where $F$ is free, be a mCM approximation of $k$. Since the ring is nonregular, $A$ is not free. Thus $\underline{p} : A \to k$ is a nonzero subobject of $k$. Since $(k,P) = 0$ for any free $P$, isomorphisms with domain $k$ lift modulo projectives. Thus, if $\underline{p}$ were an isomorphism, the same would be true for $p$, an obvious contradiction.}
\end{example}

Now we want to look at monomorphisms in the stable category with a fixed codomain and show that, for certain choices of the codomain, all monomorphisms are split. For that, we will use the following result of Hilton-Rees~\cite[Theorem 2.5]{HR} characterizing c-epimorphisms between covariant $\mathrm{Ext}$-functors.

\begin{theorem}\label{hr-epi}
 Let $\mathbf{C}$ be an abelian category with enough projectives, and $f : A \to B$ a morphism in $\mathbf{C}$ such that the induced natural transformation 
 $\mathrm{Ext}^{1}(B,-) \to \mathrm{Ext}^{1}(A,-)$ is a c-epimorphism. Then there are an object $E$ of $\mathbf{C}$, a projective object $Q$ of $\mathbf{C}$ and an exact
 sequence 
 \[
\xymatrix
	{
	0 \ar[r] 
	& Q \ar[r]
	& A \oplus E \ar[r]^>>>>>{f \perp g}
	& B \ar[r]
	& 0
	} 
\] 
Further $Q$ may be taken to be any projective object of $\mathbf{C}$ such that there exists an epimorphism $Q \to A$. \qed
\end{theorem}

The promised result can be stated as follows.

\begin{proposition}
 Let $\Lambda$ be an arbitrary ring and $B$ a $\Lambda$-module such that the functor 
$\mathrm{Ext}^{1}(B,-)$ is projectively stable (i.e., vanishes on projectives). Then any monomorphism in the stable category with codomain $B$ splits.
\end{proposition}

\begin{proof}
 Suppose $f : A \to B$ is such that $\underline{f}$ is a monomorphism. We want to show that $\underline{f}$ splits. Without loss of generality, we may assume that $f$ is epic, so that we have a short exact sequence 
 $0 \to \Ker f \overset{\ker f}{\to} A \overset{f}{\to} B \to 0$. By Proposition~\ref{mono}, 
 $\ker f$ factors through a projective. Passing to the corresponding long cohomology c-exact sequence of functors, we have that 
 $\mathrm{Ext}^{1}(f,-) : \mathrm{Ext}^{1}(B,-) \to \mathrm{Ext}^{1}(A,-)$ is a c-epimorphism. By Theorem~\ref{hr-epi}, we have a short exact sequence 
 $0 \to Q \to A \oplus E \overset{f \perp g}{\to} B \to 0$ with $Q$ projective. By assumption, this sequence splits and, therefore, $f \perp g$ induces an isomorphism modulo projectives. Since $f = (f \perp g) i$, where $i : A \to A \oplus E$ is the canonical split injection, we have that $\underline{f}$ is the composition of a split monomorphism and an isomorphism, and is thus a split monomorphism.
\end{proof}

Notice that the functor $\mathrm{Ext}^{1}(B,-)$ is projectively stable for each $B$ if and only if $\Lambda$ is quasi-Frobenius. In particular, we have

\begin{corollary}
 If $\Lambda$ is quasi-Frobenius, then  any monomorphism in the stable category of 
 $\Lambda$ splits.\footnote{This also follows at once from the well-known fact that the stable category of a quasi-Frobenius ring is triangulated.} \qed
\end{corollary}

Now we focus on monomorphisms in the stable category of a left hereditary ring.

\begin{lemma}\label{basic}
Let $\Lambda$ be a left hereditary ring, and  $f:A\rightarrow B$ a
homomorphism of $\Lambda$-modules. The following are equivalent:

\begin{enumerate}
 \item $\underline{f} = 0$;
\smallskip

\item there are submodules $K$ and $P$ of $A$, where $P$ is projective, such that  $A = K \oplus P$, and a commutative diagram 

\[
 \xymatrix
	{
	K \oplus P \ar[rr]^{f} \ar@{->>}[rd]^{p} & &B \\
	& P \ar[ru]
	}
\]
where $p$ is the canonical projection.

\end{enumerate}
\end{lemma}

\begin{proof} (1) $\Rightarrow$ (2). If $f$ factors through a projective, then, since $\Lambda$ is left hereditary, the image of $A$ in that projective is itself projective, and we take that image (or, more precisely, its image in $A$ under a splitting) for $P$. The rest is clear.
\smallskip

The implication (2) $\Rightarrow$ (1) is trivial.
\end{proof}

Next we characterize left hereditary rings by properties of their stable categories.

\begin{theorem}\label{characterization}
 Let $\Lambda$ be a ring. The following are equivalent:
 
\begin{enumerate}
 \item $\Lambda$ is left hereditary;
 \smallskip
 
 \item given a homomorphism $f : A \to B$ of $\Lambda$-modules, $f$ factors through a projective if and only if $f : A \to f(A)$ does;
 \smallskip
 
 \item given a homomorphism $f : A \to B$ of $\Lambda$-modules, $\underline{f}$ is a monomorphism in the stable category if and only if $\mathrm{ker}f$ factors through a projective;
 \smallskip
 
 \item given a homomorphism $f : A \to B$ of $\Lambda$-modules, $\underline{f}$ is a monomorphism in the stable category if and only if $\mathrm{Ker} f$ is projective;
 \smallskip
 
 \item the quotient functor $\mathcal{Q}$ preserves monomorphisms;
 \smallskip
 
 \item for any module $X$, the functor $\h^{X}$ is left-exact;
 \smallskip
 
 \item the quotient functor $\mathcal{Q}$ preserves kernels;
 \smallskip 
 
\item the quotient functor $\mathcal{Q}$ preserves finite limits.
 
\end{enumerate}
\end{theorem}
 
\begin{proof}
 (1) $\Rightarrow$ (2). The ``if'' part is trivial and is true for any ring.
 If $A \overset{f_{1}}{\to} P \overset{f_{2}}{\to} B$ is a factorization of $f$ with $P$ projective, then, by Lemma~\ref{basic}, we may assume that~$f_{1}$ is onto. But in that case, $f(A) = f_{2}(P)$, and therefore $A \overset{f_{1}}{\to} P \overset{f_{2}}{\to} f(A)$ is the desired factorization of $f : A \to f(A)$.
 \smallskip
 
 (2) $\Rightarrow$ (3). 
 The ``only if'' is true by Proposition~\ref{mono}, (a).
 
 Conversely, suppose that the kernel of $f$ factors through a projective.  Assuming that $g:C \rightarrow A$ is a module homomorphism such that $fg$ factors through a projective, we want to show that, in the stable category, $\underline{g} = 0$. 
By (2), the map $fg : C \to fg(C)$ also factors through a projective. As $fg(C) \subset f(A)$, the same is true for $fg : C \to f(A)$. But  $f: A  \to f(A)$, being an epimorphism with kernel $\ker f$ represents, by Proposition~\ref{mono}, (b) a monomorphism in the stable category. Hence $\underline{g} = 0$.
\smallskip

(3) $\Rightarrow$ (4). 
 The ``if'' part is immediate. For the ``only if'' part, assume that  $\underline{f}$ is a monomorphism. By (3), $\mathrm{ker}f$ factors through a projective: 
 $\mathrm{Ker}f \overset{i}{\to} P \to A$. Since $\mathrm{ker}\,i$ is a zero map, (3) implies that $\underline{i}$ is a monomorphism. Since $P$ is zero in the stable category, $\mathrm{Ker}f$ has to be projective, too.
 \smallskip
 
 (4) $\Rightarrow$ (5).
 Trivial.
 \smallskip
 
 (5) $\Rightarrow$ (6)
 By Proposition~\ref{half-exact}, $\h^{X}$ is half-exact. Thus 
 $\h^{X}$ is left-exact for all $X$ if and only if $\h_{f}$ is a c-monomorphism for any monomorphism $f$ of $\Lambda$-modules. But  $\h_{f} = h_{\underline{f}}$, which is a c-monomorphism since $\underline{f}$ is a monomorphism by assumption.
 \smallskip
  
 (6) $\Rightarrow$ (7). Let $\alpha$ be a monomorphism of 
 $\Lambda$-modules. By the left-exactness, $\h^{X}({\alpha}) = h^{X}(\underline{{\alpha}})$ is a monomorphism for any module 
 $X$, which means that  $\underline{\alpha}$ is a monomorphism, and therefore $\mathcal{Q}$ preserves monomorphisms. This, together with the epi-mono factorization of an arbitrary homomorphism $f : A \to B$, shows that $\underline{f} : A \to B$ has a kernel if and only if $\underline{f} : A \to f(A)$ does, and when they do, the two kernels are isomorphic. Thus, we may assume that $f$ is an epimorphism. 
 
 We want to show that $\underline{\ker f}$ is a kernel of 
 $\underline{f}$ in the stable category. This means that  
 $\underline{\mathrm{ker}f}$ has a universal lifting property for the morphisms $X \to A$ whose compositions with $\underline{f}$ are zero.  The existence of a lifting follows from the half-exactness of 
 $\h^{X}$ applied to the short exact sequence  
 $0 \to \Ker f \overset{\ker f}{\to} A \overset{f}{\to} B \to 0$. The uniqueness of a lifting follows from the fact that  
 $\underline{\mathrm{ker}f}$ is a monomorphism. 
 \smallskip

(7) $\Rightarrow$ (8). As we mentioned above, $\mathcal{Q}$ preserves biproducts, and hence preserves finite products. If it also preserves kernels, it preserves equalizers. Then, as is well-known, it preserves finite limits.

(8) $\Rightarrow$ (1). Since $\mathcal{Q}$ preserves finite limits, it preserves equalizers, and therefore it preserves kernels. Suppose $\Lambda$ is not left hereditary. Then there is a projective module $P$ containing a non-projective submodule $A$. If $f$ is the inclusion of $A$ in $P$, then  $\ker f = 0$, whereas $\ker \underline{f}$ is not zero because $P$ is projective and $A$ is not, a contradiction. 
\end{proof}
 
 The just proved theorem immediately implies

\begin{proposition}\label{kernel}
The stable category of a left hereditary ring has kernels, and hence is finitely complete. \qed
\end{proposition}

\section{Epimorphisms in the stable category}\label{StableEpis}
We begin by general remarks about epimorphisms in the stable categories. Their proofs are analogous to those of their counterparts in the previous section, and are therefore omitted.

\begin{lemma}\label{epi}
 Let $\Lambda$ be a ring and $f : A \to B$ a homomorphism of $\Lambda$-modules. If  $\underline{f}$ is an epimorphism in the stable category, then $\mathrm{coker}f$ factors through a projective. If, in addition, $A$ is projective, then so is $B$.  \qed
\end{lemma}

\begin{lemma}
Let $\Lambda$ be a ring and $f : A \to B$ a homomorphism of $\Lambda$-modules. 
If~$B$ is totally stable or if $\Lambda$ is left hereditary and $B$ is stable, then the following are equivalent: 
\begin{enumerate}
 \item $\underline{f}$ is a split epimorphism;
 \smallskip
 \item $f$ is a split epimorphism. \qed
\end{enumerate} 
\end{lemma}

We also have the following obvious but useful fact.

%
%

\begin{lemma}\label{lift}
Any homomorphism factoring through a projective lifts over any epimorphism co-terminal with it. \qed
\end{lemma}

Now we want to characterize split epimorphisms in the stable category. 


\begin{proposition}\label{prop-epi}
 Let $\Lambda$ be a ring and $f : A \twoheadrightarrow B$ an epimorphism of 
 $\Lambda$-modules. The following are equivalent:

\begin{enumerate}
 \item $\underline{f}$ is a split epimorphism;
 \smallskip

\item $f$ is a split epimorphism;

\end{enumerate}
\end{proposition}

\begin{proof}
We only need to show (1) $\Rightarrow$ (2). Assume that $\underline{f}~\underline{g}=\underline{1_{B}}$ for some homomorphism $g : B \to A$. Then $fg=1_{B}+h$, for some $h$ such that $\underline{h} = 0$. By Lemma~\ref{lift}, there is $g_{1}: B \rightarrow A$ such that $fg_{1}=h$. Then $f(g-g_{1})=1_{B}$. 
\end{proof}

\begin{proposition}
 Let $\Lambda$ be a ring and $f : A \rightarrow B$ an epimorphism of 
 $\Lambda$-modules. The following are equivalent:

\begin{enumerate}
\item $\underline{f}$ is an isomorphism;
\smallskip

\item $f$ is a split epimorphism and $\Ker f$ is  projective.

\end{enumerate}
\end{proposition}

\begin{proof}
(1) $\Rightarrow$ (2). By Proposition~\ref{prop-epi}, $f$ is a split epimorphism. Thus $\mathrm{ker}\, f$ is a split monomorphism. On the other hand, $\underline{f}$ is a monomorphism, and so
$\mathrm{ker}\, f$ has to factor through a projective. Thus 
$\Ker f$ embeds in a projective module as a direct summand, and is therefore projective.
\smallskip

(2) $\Rightarrow$ (1). We have a split exact sequence 
$0 \to \Ker f \to A \to B \to 0$. The result now follows from 
Lemma~\ref{split}.
\end{proof}

Our next goal is to provide a necessary and sufficient condition for an epimorphism in the module category to represent an epimorphism in the stable category. First, we need new notation. 

Given an epimorphism $f : Y \to Z$ of $\Lambda$-modules, the  associated short exact sequence $0 \to X \to Y \overset{f}{\to} Z \to 0$ will be denoted by $[f]$\label{brackets}. Given any $h : Y \to Y'$, the pushout of $(f, h)$ results in a commutative diagram with exact rows:
\[
\xymatrix
	{
	0 \ar[r] 
	& X \ar[r] \ar[d] 
	& Y \ar[d]^{h} \ar[r]^{f} 
	& Z \ar[d] \ar[r] 	
	& 0
\\
	0 \ar[r] 
	& X' \ar[r]  
	& Y'  \ar[r]^{h \ulcorner^{f}} 
	& Z' \ar[r] 
	& 0
	}
\] 
where the epimorphism in the bottom row is denoted by
$h \ulcorner^{f}$. Let 
\[
\langle h \ulcorner^{f} \rangle : [f] \to [h \ulcorner^{f}]
\]
denote the chain map given by the vertical arrows. 

Secondly, we mention 

\begin{lemma}\label{0-homotopy}
Let $\Lambda$ be an arbitrary ring, and $f:A \twoheadrightarrow B$ an epimorphism of $\Lambda$-modules. Given a homomorphism 
$h:A\rightarrow Y$, the following are equivalent:

\begin{enumerate}
\item the chain map $\langle h \ulcorner^{f}\rangle$ is null-homotopic;

\item any chain map
\[
 \xymatrix
    {
    0 \ar[r] 
    & \Ker f \ar[r]^{i} \ar[d]
    & A \ar[d]^{h} \ar[r]^{f} 
    & B \ar[d]\ar[r] & 0
\\
    0 \ar[r] 
    & M \ar[r]  
    & Y  \ar[r] 
    & N \ar[r] 
    & 0
    }
\]
with $h$ in the middle and where the bottom row is exact, is null-homotopic;
\item there is a diagonal homomorphism in the pushout square of $(h, f)$ lifting the arrow parallel to $h$ over $h \ulcorner^{f}$:
\[
\xymatrix
	{
	A \ar[d]_{h} \ar[r]^{f} 
	& B \ar[d] \ar@{.>} [ld]
\\
	Y \ar[r]^{h \ulcorner^{f}}
	& D
	}
\] 

\end{enumerate}
\end{lemma}

\begin{proof}
 (1) $\Rightarrow$ (2).
Taking a pushout of ($h, f$), we have a commutative diagram 
\[
\xymatrix
	{
	& \Ker f \ar[rr]^{i} \ar[dl] \ar ' [d] [dd] 
	&
	& A \ar[dl]_{h} 	\ar ' [d]^{h} [dd] \ar[rr]^{f} 
	& 
	& B \ar[ld] \ar[dd]
\\
	X' \ar@{.>}[dr]^{b'} \ar[rr] 
	&
	& 	Y \ar@{=}[dr] \ar[rr]^>>>>>{h \ulcorner^{f}}
	&
	& Z'\ar@{.>}[dr]^{b} 
	&
\\
	& M \ar[rr]^{i'}
	&
	& Y \ar[rr] 
	&
	& N
	}
\]
of solid arrows, where the rows are short exact sequences. By the universal property of pushouts, we have a dotted arrow $b$, and therefore a dotted arrow $b'$, making the squares incident with $1_{Y}$ and the triangle incident with~$b$ commute. Since $i'$ is a monomorphism, the triangle incident with $b'$ also commutes. Since the middle triangle is trivially commutative, we have that the chain map in question factors through $\langle h \ulcorner^{f}\rangle$, and is thus null-homotopic.
\smallskip

The implication (2) $\Rightarrow$ (3) and (3) $\Rightarrow$ (1) are trivial.
\end{proof}

Now we can establish the promised criterion.

\begin{theorem}\label{char-epi}
Let $\Lambda$ be any ring and $f:A \twoheadrightarrow B$ an epimorphism of  $\Lambda$-modules. The following conditions are equivalent:

\begin{enumerate}
 \item $\underline{f}$ is an epimorphism;
 \smallskip
 
 \item the equivalent conditions of Lemma~\ref{0-homotopy} are satisfied for any $h : A \to Q$ with a projective $Q$;
 \smallskip
 
 \item the equivalent conditions of Lemma~\ref{0-homotopy} are satisfied for any $h : A \to Y$ with  $\underline{h} = 0$.
 \suspend{enumerate}
 If $\Lambda$ is left hereditary, then these conditions are equivalent to 
 \resume{enumerate}
 \item the equivalent conditions of Lemma~\ref{0-homotopy} are satisfied for any epimorphism $h : A \twoheadrightarrow Q$ with a projective $Q$.
\end{enumerate}
\end{theorem}

\begin{proof}
 (1) $\Rightarrow$ (2). Suppose $\underline{f}$ is an epimorphism. Then, in the pushout diagram
 \[
\xymatrix
	{
	A \ar@{->>}[r]^{f} \ar[d]_{h} 
	& B \ar@{..>}[dl]_{s}  \ar[d]^{h'}
\\
	Q \ar@{->>}[r]_{h \ulcorner^{f}} 
	& D
	}
\] 
of solid arrows, $h'$ factors through a projective. Since 
$h \ulcorner^{f}$ is an epimorphism, by Lemma~\ref{lift}, $h'$ lifts over it to a dotted arrow $s$.
\smallskip

(2) $\Rightarrow$ (3). This follows from the transitivity of pushouts.  
\smallskip

(3) $\Rightarrow$ (1). Under the assumptions, the equivalent conditions of Lemma~\ref{0-homotopy} are satisfied for any $h$ with a projective codomain. Let $g : B \to Z$ be a homomorphism such that $gf=h'' h$ with a projective codomain $Q$ of $h$. Taking a pushout of $(h,f)$, we have a commutative diagram
 \[
\xymatrix
	{
	A \ar@{->>}[r]^{f} \ar[d]_{h} 
	& B \ar[d]_{h^{\prime}} \ar[ddr]^{g} 
	&
\\
	Q \ar[drr]_{h^{\prime\prime}} \ar@{->>}[r]^{h \ulcorner^{f}} 
	& D \ar[dr] 
	&
\\
	&
	& Z
	}
\] 
By assumption, $h^{\prime}$ factors through $Q$, and hence the same is true for $g$. Thus $\underline{f}$ is an epimorphism.

The last assertion of the theorem follows, once again, from the transitivity of pushouts and Lemma~\ref{basic}.
\end{proof}

The next two results immediately follow (or can be easily verified directly) from this proposition and the obvious fact that any chain map into a contractible complex is null-homotopic.

\begin{corollary}
 Let $f : A \twoheadrightarrow B$ be an epimorphism. If, for any  $h:A \to Q$ with $Q$ projective, $h(\Ker f)$ is a direct summand of $Q$, then $\underline{f}$ is an epimorphism. \qed
\end{corollary}

\begin{corollary}
Let $f:A \twoheadrightarrow B$ be an epimorphism. If, for any  $h:A \to Q$ with $Q$ projective, $Q/h(\Ker f)$ is projective, then $\underline{f}$ is an epimorphism. \qed 
\end{corollary}

Throughout the rest of this section, we assume that  $\Lambda$ is left hereditary. In that case, we can give another criterion for a homomorphism to be an epimorphism in the stable category.

\begin{theorem}\label{hered-epi}
Let $\Lambda$ be a left hereditary ring and $f:A\rightarrow B$ a homomorphism of $\Lambda$-modules. 
The following are equivalent:
\begin{enumerate}
 \item $\underline{f}$ is an epimorphism;
 \smallskip
 \item for any representation $A = K \oplus P$ with a projective~$P$, there exists a representation $B = M \oplus Q$ with a projective $Q$ such that $f(K) \supset M$.
\end{enumerate}
\end{theorem}

\begin{proof} (1) $\Rightarrow$ (2). Let $\underline{f}$ be an epimorphism and $A = K \oplus P$ with a projective~$P$. Since the inclusion $K \to K \oplus P$ is an isomorphism in the stable category, the restriction of $f$ to $K$ yields an epimorphism in the stable category. By the defining property of epimorphisms, the cokernel $B \overset{\pi}{\to} B/f(K)$ of that inclusion 
factors through some projective $Q$. By Lemma~\ref{basic}, we have $B = M \oplus Q$ and a commutative diagram
\[
\xymatrix
	{ 
	M \oplus Q \ar[d] \ar[r]^>>>>>{\pi} & (M \oplus Q)/f(K)\\
	Q \ar[ur] &
	}
\] 
where the vertical map is the canonical projection. As a result, we have a commutative diagram of solid arrows with exact rows 
\[
\xymatrix
	{
	0 \ar[r] & M \ar@{..>}[d] \ar[r]    & M \oplus Q \ar@{=}[d] \ar[r] 
	& Q \ar[r] \ar[d]               & 0\\
	0 \ar[r] & f(K) \ar[r] & M \oplus Q \ar[r]^>>>>>{\pi} 
	& (M \oplus Q)/f(K) \ar[r] & 0
	}
\]
Completing it with a dotted arrow, we have the required inclusion $M \subset f(K)$.
\smallskip

(2) $\Rightarrow$ (1). Assume that  $\underline{g} \underline{f} = 0$, for some  $g:B\rightarrow C$. Then $gf$ factors through a projective and, by Lemma~\ref{basic}, we have $A = K \oplus P$ for some $K$ and a projective~$P$. Let 
$M \oplus Q$ be a decomposition of $B$ such that  $f(K) \supset M$. We then have a commutative diagram
 \[
\xymatrix
	{
	K \oplus P \ar[d] \ar[r]^{f} & M \oplus Q \ar[d]^{g} \\
	P \ar[r] & C
	}
\] 
where the left vertical map is the canonical projection. Since $g$ vanishes on $f(K)$, it also vanishes on $M$ and thus factors through $M \oplus Q \to (M \oplus Q)/M \simeq Q$.
\end{proof}

\begin{corollary}\label{epi-cases}
Let $\Lambda$ be a left hereditary ring and $f:A\rightarrow B$ a homomorphism of $\Lambda$-modules.

\begin{enumerate}
\item[(a)]\label{a} If $A$ is stable and $f$ is an epimorphism, then $\underline{f}$ is
an epimorphism;
\smallskip

\item[(b)] if $\underline{f}$ is an epimorphism, and $B$ is stable, then
$f$ is an epimorphism;
\smallskip

\item[(c)] if both $A$ and $B$ are stable, then $\underline{f}$ is an
epimorphism if and only if $f$ is an epimorphism.
\end{enumerate}

\end{corollary}

\begin{proof}
(a) is obvious. For (b), take the representation $A = A \oplus \{0\}$. Since $\underline{f}$ is an epimorphism, we have 
$B = M \oplus Q$ for some $M\subset f(A)$ and projective $Q$. As $B$ is stable, $B = M = f(A)$. (c) is immediate.
\end{proof}

\begin{proposition}\label{mono-epi}
 Let $\Lambda$ be a left hereditary ring and $f:A\rightarrow B$ a  monomorphism of $\Lambda$-modules. The following  are equivalent:
\begin{enumerate}
 \item $\underline{f}$  is an epimorphism;
 \smallskip
 
 \item there is a representation $B = M \oplus Q$ with $Q$    projective and $M\subset f(A)$;
 \smallskip

\item $\underline{f}$ is a split epimorphism;
\smallskip

\item $\underline{f}$  is an isomorphism.
\end{enumerate}
\end{proposition}

\begin{proof}
The implication (1) $\Rightarrow$ (2) immediately follows from 
Theorem~\ref{hered-epi}.
\smallskip

(2) $\Rightarrow$ (3). By Lemma~\ref{split}, the inclusion 
$M \to M \oplus Q$ is an isomorphism in the stable category. Therefore, in the commutative square 

\[
\xymatrix
	{
	M \ar[d] \ar@{=}[r] & M \ar[d] \\
	f(A) \ar[r] & M \oplus Q
	}
\] 
of canonical inclusions, the bottom map is a split epimorphism. Composing it with the isomorphism $A \to f(A)$ (since $f$ is a monomorphism), we have that $\underline{f}$ is a split epimorphism.
\smallskip

(3) $\Rightarrow$ (4). By Theorem~\ref{characterization}, 
$\underline{f}$ is a monomorphism. If it is also a split epimorphism, it obviously is an isomorphism.
\smallskip

(4) $\Rightarrow$ (1) is trivial.
\end{proof}

We can now give a criterion for a morphism in the stable category of a left hereditary ring to be a split epimorphism.
 
\begin{theorem}\label{3.12}
Let $\Lambda$ be a left hereditary ring and $f : A \rightarrow B$ a homomorphism of $\Lambda$-modules. The following are equivalent:

\begin{enumerate}
 \item $\underline{f}$ is a split epimorphism;
 \smallskip 
 \item $\Ker f$ is a direct summand of $A$ and there is a projective module $Q$ such that $B \simeq M \oplus Q$ and $M \subset f(A)$.
\end{enumerate}
\end{theorem}

\begin{proof}
 (1) $\Rightarrow$ (2). Since $\underline{f}$ is a split epimorphism, the epi-mono factorization 
 $A \overset{e}{\to} f(A) \overset{m}{\to} B$ of $f$ shows that  the same is true for 
 $\underline{m}$. By  Theorem~\ref{characterization}, the latter is also a monomorphism, and therefore an isomorphism. Using the fact that $\underline{f}$ is a split epimorphism again, we have that the same is true for $\underline{e}$. Applying Proposition~\ref{prop-epi}, we have that $e$ is a split epimorphism, and therefore 
$\Ker f$ is a direct summand of $A$. Applying 
Proposition~\ref{mono-epi} to the monomorphism $m$, we have the desired decomposition of $B$.

(2) $\Rightarrow$ (1). Since $\Ker f$ is a direct summand of $A$, the homomorphism $e$ in the epi-mono factorization $f = me$ is a split epimorphism, and therefore so is $\underline{e}$. Applying 
Proposition~\ref{mono-epi} to the monomorphism $m$, we have that $\underline{m}$ is also a split epimorphism. Thus, 
$\underline{f}$ is a split epimorphism, too.
\end{proof}

As we mentioned in the introduction, for any ring, one has the well-known criterion of Heller for a morphism in the stable category to be an isomorphism. Using Theorems~\ref{characterization} and~\ref{3.12}, we have another criterion in the case of a left hereditary ring.

\begin{theorem}\label{iso}
Let $\Lambda$ be a left hereditary ring and $f : A \rightarrow B$ a homomorphism of $\Lambda$-modules. The following are equivalent:

\begin{enumerate}
 \item $\underline{f}$ is an isomorphism;
 \smallskip
 
 \item $\Ker f$ is a projective summand of $A$ and there is a projective module $Q$ such that $B = M \oplus Q$ and $M \subset f(A)$. \qed
\end{enumerate} 
\end{theorem}

\section{Epimorphisms in the stable category and torsion}\label{epis-and-torsion}
We continue to assume that $\Lambda$ is an arbitrary ring and $A$ an $\Lambda$-module. Recall that the torsion submodule $t(A)$ is defined as the kernel of the canonical map $e_{A} : A \to A^{\ast\ast}$. Equivalently, this is the intersection of the kernels of all linear forms on $A$.\footnote{A more appropriate name for this concept is 1-torsion, see, for example,~\cite[p. 2597]{Mar}. For a finite module over a commutative domain this definition is equivalent to the classical definition of torsion. Without the finiteness assumption, the two are not equivalent: the $\Z$-module $\Q$, being divisible, coincides with its 1-torsion submodule, but has no torsion in the classical sense. Notice that the classical torsion for modules over commutative domains is always contained in 1-torsion.} Let~$A^{\sharp}$ denote the image of $A$ under $e_{A}$. The naturality of $e_{A}$ shows that $t(-)$ is a subfunctor of the identity functor on $\Lambda$-$\Mod$, with  quotient $(-)^{\sharp}$. Notice that $t(A) = A$ if and only if $A^{\ast}=0$. 
 
The next series of results shows that, to a large extent, the study of epimorphisms in the stable category can be handled ``modulo'' torsion. The following lemma is a direct consequence of the definitions.

\begin{lemma}\label{torsion}
 Any homomorphism that factors through a projective vanishes on the torsion submodule. \qed
\end{lemma}

\begin{proposition}\label{kernel-in-torsion}
 Let $f : A \twoheadrightarrow B$ be an epimorphism of $\Lambda$-modules such that $\Ker f$ is contained in the torsion submodule of $A$. Then $\underline{f}$ is an epimorphism.
\end{proposition}

\begin{proof}
 Suppose $g : B \to L $ is such that $gf$ factors through a projective. We then have a commutative diagram of solid arrows
 \[
\xymatrix
	{
	0 \ar[r] & \Ker f \ar[r] & A \ar[d]_{h} \ar[r]^{f} & B 	\ar@{..>}		[dl]_{h_{1}} \ar[d]^{g} \ar[r] 	& 0\\
	& & P \ar[r]_{h_{2}} & L &
	}
\] 
where the top row is exact and $P$ is projective. By assumption,  $h$ restricts to zero on $\Ker f$ and, therefore, factors through $f$ by way of a dotted arrow $h_{1}$. Since $f$ is an epimorphism, $g = h_{2} h_{1}$, showing that $\underline{g} = 0$.
\end{proof}

\begin{corollary}\label{modtorsion}
 Let $p : A \to A/t(A) \simeq A^{\sharp}$ be the canonical map. Then $\underline {p}$ is an epimorphism in the stable category.
 \qed
\end{corollary}

\begin{corollary}
If $f : A \twoheadrightarrow B$ is an epimorphism and $A^{\ast} = 0$, then $\underline{f}$ is an epimorphism in the stable category. 
\end{corollary}

\begin{proof}
 In this case, $t(A) = A$.
\end{proof}

\begin{corollary}\label{i-epi}
 In the above notation, let $i : t(A) \to A$ be the inclusion map. The following are equivalent:
 
\begin{enumerate}
 \item $\underline{i}$ is an isomorphism in the stable category;
 \smallskip
 \item $\underline{i}$ is an epimorphism in the stable category;
 \smallskip
 \item $A^{\sharp}$ is projective.
\end{enumerate}
\end{corollary}
\begin{proof}
 (1) $\Rightarrow$ (2). Immediate.
 \smallskip
 
 (2) $\Rightarrow$ (3). Since $\underline{i}$ is an epimorphism and $p$ is a cokernel of $i$, $\underline{p} = 0$. By Corollary~\ref{modtorsion}, $\underline{p}$ is an epimorphism, and therefore $A^{\sharp}$ is zero in the stable category.
 \smallskip

(3) $\Rightarrow$ (1). This follows at once from Lemma~\ref{split}.
\end{proof}

\begin{proposition}\label{middle-to-right}
 Let $f : A \to B$ be a homomorphism of $\Lambda$-modules. If  $\underline{f}$ is an epimorphism, then so is $\underline{f^{\sharp}}$, and  $f(A) \supset t(B)$. 
\end{proposition}

\begin{proof}
We have a commutative diagram 
\[
\xymatrix
	{
	0 \ar[r] & t(A) \ar[d]^{t(f)} \ar[r] & A \ar[d]^{f} \ar[r] & 		A^{\sharp} \ar[d]^{f^{\sharp}} \ar[r] & 0\\
	0 \ar[r] & t(B) \ar[r]^{i} & B \ar[r]^{p} & B^{\sharp} \ar[r] & 0\\
	}
\] 
with exact rows. By Corollary~\ref{modtorsion}, $\underline{p}$ is an  epimorphism. Since $\underline{f}$ is an epimorphism, so is $\underline{f^{\sharp}}$. To prove the second assertion, notice that $\mathrm{coker} f$ factors through a projective and, by Lemma~\ref{torsion}, must vanish on $t(B)$.
\end{proof}

\begin{remark}\label{endo-stable}
 Since the endofunctor $(-)^{\sharp}$ preserves projectives, by the universal property of the quotient functor $\mathscr{Q}$, we have an endofunctor on the stable category. 
 The just proved proposition shows that that endofunctor preserves epimorphisms.
\end{remark}

\begin{corollary}
 In the above notation, if $A^{\sharp}$ is projective and $B^{\sharp}$ is not, then $\underline{f}$ is not an epimorphism in the stable category.\footnote{In particular, this result applies when $t(A) = A$ and $B^{\sharp}$ is not projective.}
\end{corollary}

\begin{proof}
Under the above assumptions, $\underline{f^{\sharp}}$ cannot be an epimorphism.\footnote{Alternatively, this result can also be deduced from Corollary~\ref{i-epi}. Since $A^{\sharp}$ is projective, the inclusion $t(A) \to A$ is an isomorphism in the stable category. Assuming that 
$\underline{f}$ is an epimorphism, we have that the inclusion 
$t(B) \to B$ is an epimorphism in the stable category (see the diagram in Proposition~\ref{middle-to-right}). Then, by Corollary~\ref{i-epi}, $B^{\sharp}$ is projective, a contradiction.}
\end{proof}

The next result gives a partial converse to Proposition~\ref{middle-to-right}.

\begin{proposition}
 Let $f : A \to B$ be a homomorphism of $\Lambda$-modules. If  $\underline{f^{\sharp}}$ is an epimorphism in the stable category and $t(f)$ is an epimorphism, then $\underline{f}$ is an epimorphism in the stable category.
\end{proposition}

\begin{proof}
Let $g : B \to L$ be such that $gf$ factors through a projective; we want to show that so does $g$.
In the commutative diagram 
\[
\xymatrix
	{
	0 \ar[r] 
	& t(A) \ar[dd]^{t(f)} \ar[r] 
	& A \ar[dd]_{f} \ar[rd]^{h} \ar[rr]^{q} 
	& 
	& A^{\sharp} \ar@{..>}[ld] \ar[dd]^{f^{\sharp}} \ar[r] 	
	& 0\\
	&
	&
	& P \ar[ldd] 
	&
	&\\
	0 \ar[r] 
	& t(B) \ar[r] 
	& B \ar[rr]^{p} \ar[d]_{g} 
	& 
	& B^{\sharp} \ar[r] \ar@{..>}[lld]_{r}  
	& 0\\
	&  
	& L 
	& 
	&
	}
\] 
of solid arrows, where $P$ is projective, $h$ vanishes on $t(A)$ and hence extends to~$A^{\sharp}$ by way of a dotted arrow $A^{\sharp} \to P$. As $gf$ vanishes on $t(A)$ 
and $t(f)$ is onto, $g$ vanishes on $t(B)$, and therefore $g$ extends to $B^{\sharp}$ by way of a dotted arrow~$r$. Composing  the epimorphism $q$ with the two compositions containing the dotted arrows, we have that the square $A^{\sharp}PLB^{\sharp}$ commutes, and therefore $r f^{\sharp}$ factors through $P$. As~$\underline{f^{\sharp}}$ is an epimorphism, $r$ factors through a projective, and so does $g$.
\end{proof}

\begin{remark}
In the terminology of Remark~\ref{endo-stable}, the just proved result says that the endofunctor on the stable category determined by $(-)^{\sharp}$ reflects epimorphisms whose representatives are epimorphisms on the corresponding torsion submodules.
\end{remark}

Now we want to look at epimorphisms in the stable category with a fixed domain and show that, for certain choices of the domain, all epimorphisms are split.

\begin{proposition}\label{sub-proj}
Let $\Lambda$ be an arbitrary ring and $A$ a submodule of a projective $\Lambda$-module. Then any epimorphism in the stable category with domain $A$ splits.
\end{proposition}

\begin{proof}
Suppose $f : A \to B$ is such that $\underline{f}$ is an epimorphism. We want to show that $\underline{f}$ splits. Without loss of generality, we may assume that $f$ is epic, so that we have a short exact sequence $0 \to \Ker f \to A \overset{f}{\to} B \to 0$. By assumption, we have a short exact sequence $0 \to A \overset{h}{\to} Q \overset{q}{\to} C \to 0$ with $Q$ projective. Taking the pushout of $(h,f)$ we have a commutative diagram 
\[
\xymatrix
	{
	&
	& 0 \ar[d]
	& 0 \ar[d]
	&
\\
	0 \ar[r]
	& \Ker f \ar[r] \ar@{=}[d]
	& A \ar[r]_{f} \ar[d]_{h}
	& B \ar[r] \ar[d]^{h'} \ar@{.>}[dl]_{s} \ar @/_/@{.>}[l]_{s'}
	& 0
\\
	0 \ar[r]
	& \Ker h \ulcorner^{f} \ar[r]
	& Q \ar[r]^{h \ulcorner^{f}} \ar[d]_{q}
	& D \ar[r] \ar[d]
	& 0
\\
	&
	& C \ar@{=}[r] \ar[d]
	& C \ar[d]
	&
\\
	&
	& 0
	& 0
	&
	}
\] 
of solid arrows with exact rows and columns. By Theorem \ref{char-epi}, there is a dotted arrow $s$ such that $h' = h \ulcorner^{f} s $. It is clear that the square $ABDQ$ is a pullback. Thus, there exists a homomorphism $s' : B \to A$ such that $hs' = s$ and
$f s' = 1_{B}$. In particular, $f$ splits, and so does $\underline{f}$.
\end{proof}

Notice that each $\Lambda$-module is a submodule of a projective module if and only 
if~$\Lambda$ is quasi-Frobenius. In particular, we have

\begin{corollary}
 If $\Lambda$ is quasi-Frobenius, then any epimorphism in the stable category of $\Lambda$ splits.\footnote{This also follows at once from the well-known fact that the stable category of a quasi-Frobenius ring is triangulated.} \qed
\end{corollary} 

Since submodules of projectives are torsionfree, it is natural to ask whether the conclusion of Proposition~\ref{sub-proj} holds for arbitrary torsionfree modules. That this is not the case for the larger class of classically torsionfree modules over commutative domains can be seen from the following simple

\begin{example}
Consider the short exact sequence $0 \to \Z \to \Q \overset{f}{\to} \Q/\Z \to 0$ in the category of abelian groups. We claim that $\underline{f}$ is an epimorphism. Suppose we have a commutative diagram 
 \[
\xymatrix
	{
	\Q \ar[d]^{h} \ar[r]^{f}
	& \Q/\Z \ar[d]^{g}
\\
	Q \ar[r]
	& D
	}
\] 
where $Q$ is projective. Since $\Q$ is divisible, $h = 0$. Since $f$ is an epimorphism, $g=0$, i.e., $\underline{f}$ is indeed an epimorphism. On the other hand, by Proposition~\ref{prop-epi}, $\underline{f}$ is not split.
\end{example}

Our final results in this section deal with the full subcategory of torsionfree modules. It it well-known that it is reflective in 
$\Lambda$-$\Mod$, i.e., the inclusion functor has a left adjoint, called a \texttt{reflector}, which is given by $(-)^{\sharp} = (-)/t(-)$. We want to show that the full subcategory determined by torsionfree modules is also reflective in $\Lambda$-$\underline{\Mod}$.

Let $\mathbf{X}$ be a reflective subcategory of 
$\Lambda$-$\Mod$ with inclusion $i :\mathbf{X} \rightarrow \Lambda$-$\Mod$ and reflector $r : \Lambda \text{-}\Mod \rightarrow\mathbf{X}$. Then $\mathbf{X}$ is additive~\cite[Proposition 3.5.4]{B}. Since adjoint functors between additive categories are additive, we have that both $i$ and $r$ are additive. Assume that $ir$ sends projectives to projectives. Let $\mathbf{X}^{e}$ denote the full subcategory of the stable category determined by ${\mathbf{X}}$; it is clearly additive. By the universal property of the quotient functor, $r$ extends to an additive functor $r^{e} : \Lambda \text{-}\underline{\Mod} \to \mathbf{X}^{e}$. Finally, let~$i^{e}$ be the inclusion $\mathbf{X}^{e} \to \Lambda\text{-}\underline{\Mod}$. We now have a commutative diagram 
\[
\xymatrix
	{
	 \mathbf{X} \ar[r]^{i} \ar[d]^{\mathscr{Q |_{\mathbf{X}}}}
	 & \Lambda \text{-} \Mod \ar[r]^{r} \ar[d]^{\mathscr{Q}}
	 & \mathbf{X} \ar[d]^{\mathscr{Q |_{\mathbf{X}}}}
\\
	\mathbf{X}^{e} \ar[r]^{i^{e}} 
	& \Lambda \text{-}\underline{\Mod} \ar[r]^{r^{e}}
	& \mathbf{X}^{e}
	}
\]

\begin{lemma} 
$(r^{e}, i^{e})$ is an adjoint pair, and hence $\mathbf{X}^{e}$ is reflective in $\Lambda \text{-} \underline{\Mod}$.
\end{lemma}

\begin{proof}
Let $\eta$ be the unit of the adjunction $r \dashv i$, and $\varepsilon$ its counit (which is an isomorphism). Then the compositions 
%
%
%
%
\[
\xymatrix
	{
	i \ar[r]^{\eta i} 
	& i r i \ar[r]^{i \varepsilon}
	& i
	}
\]
\[
\xymatrix
	{
	r \ar[r]^{r \eta} 
	& r i r \ar[r]^{\varepsilon r} 
	& r
	}
\]
are the identities of $i$ and, respectively, $r$.
%
%
Applying $\mathscr{Q}$ to the first equality, and the restriction of $\mathscr{Q}$ to $\mathbf{X}$ to the second equality, we obtain  similar equalities for $i^{e}$ and~$r^{e}$, proving the desired assertion.
\end{proof}

Combining this with the above observation, we have 

\begin{corollary}
 The full subcategory determined by torsionfree modules is reflective in the stable category of $\Lambda$. \qed
\end{corollary}

\section{Normal monomorphisms in the stable category}\label{NormalMonos}
Recall that a monomorphism (resp., epimorphism) in a category is said to be \texttt{normal} if it is a kernel (resp., cokernel) of some morphism. For example, in any preadditive category, all split monomorphisms and split epimorphisms are normal.

\begin{lemma}\label{essential}
Let $\Lambda$ be an arbitrary ring and $A$ a $\Lambda$-module. Suppose $A\rightarrow E$ is an essential extension such that 
the corresponding morphism $\underline{p}:E \rightarrow E/A$ is a kernel, in the stable category, of $\underline{f} : E/A \rightarrow X$ for some $f : E/A \to X$. Then $f$ is a monomorphism.
\end{lemma}

\begin{proof}

Under the assumptions, $\ker f$ lifts over $p$ modulo projectives. By Lemma~\ref{lift}, $\ker f$ lifts over $p$ by way of some homomorphism $i : \Ker f \to E$. Since $\ker f$ is a mono, $i(\Ker f) \cap A = (0)$, and therefore $i(\Ker f) = 0$. Since $i$ is a monomorphism, $\Ker f = (0)$.
%
%
\end{proof}
Lemma~\ref{essential} implies

\begin{corollary}\label{inj-dim-1}
Let $\Lambda$ be an arbitrary ring, $P$ a projective $\Lambda$-module of injective dimension one, with a minimal injective resolution 
$0 \rightarrow P\rightarrow I_{0} \overset{p}\rightarrow I_{1}
\rightarrow 0$, where~$I_{0}$ is non-projective. Then $\underline{p}:I_{0}\rightarrow I_{1}$ is a non-normal monomorphism.
\end{corollary}

\begin{proof}
By Proposition~\ref{mono}, $\underline{p}:I_{0}\rightarrow I_{1}$ is a monomorphism. If it is a kernel of some morphism 
$\underline{f}:I_{1}\rightarrow X$ in the stable category, then, by
Lemma~\ref{essential}, $f$ is a monomorphism. But then it is a split monomorphism, since $I_{1}$ is injective, and hence 
$\underline{f}$ is a monomorphism. Hence its kernel is a zero object, i.e. $I_{0}$ is projective, a contradiction.
\end{proof}

\begin{corollary}\label{hered-normal}
Let $\Lambda$ be a left hereditary ring, $P$ be a projective 
$\Lambda$-module, and $P\rightarrow I$ an injective envelope. If $I$ is non-projective (equivalently, $P$ is non-injective), then the corresponding $\underline{p}:I\rightarrow I/P$ is a non-normal monomorphism.
\end{corollary}

\begin{proof}
Immediately follows from Corollary~\ref{inj-dim-1}.
\end{proof}

Recall that a category is said to be \texttt{normal} if every monomorphism is normal.\footnote{There are several variants of this definition; the one used here is taken from~\cite{Mi}.} We now have

\begin{theorem}\label{normal}
The stable category of a left hereditary ring $\Lambda$ is normal if and only if the injective envelope of $_{\Lambda}\Lambda$ is projective.
\end{theorem}

\begin{proof}
The ``only if'' part follows immediately from 
Corollary~\ref{hered-normal}. For the ``if'' part, assume that the injective envelope of $\Lambda$ is projective. Then the same is true for any projective, as is shown in~\cite[Theorem 3.2]{CR}. Suppose the class of $p : A \to B$ is a monomorphism in the stable category. We may assume, without loss of generality, that~$p$ is onto. By Theorem~\ref{characterization}, the kernel~$P$ of~$p$ is projective, and we take the quotient map $p : A \to A/P$ as a representative of the original monomorphism in the stable category. 

Let $i':P\rightarrow I$ be the injective envelope of $P$. Extending $i'$ over $i$ and using the snake lemma, we have a commutative diagram 
\[
\xymatrix
	{
	&& 0 \ar[d] & 0 \ar[d] \\
	& & \Ker f \ar[d]^{\ker f} \ar@{=}[r] & \Ker f \ar[d]^{\ker f'} \\
	0 \ar[r] & P \ar[r]^{i} \ar@{=}[d] & A \ar[r]^{p} \ar[d]^{f} 
	& A/P \ar[d]^{f'} \ar[r] & 0 \\
	0 \ar[r] & P \ar[r]^{i'} & I \ar[r]^{p'} & I/i'(P) \ar[r] & 0
	}
\]
with exact rows and columns. We claim that 
$\underline{p} = \ker \underline{f}'$.  Indeed, the diagram shows that $p \ker f = \ker f'$. By Theorem~\ref{characterization}, 
$\underline{p}\, \underline{\ker f} = \ker \underline{f}'$. As we remarked above, $I$ is projective, hence $\underline{\ker f}$ is an isomorphism, and thus  $\underline{p}$ is a kernel of 
$\underline{f}'$.
\end{proof}

Recall that a category is said to be \texttt{well-powered} if any object $B$ has only a set of subobjects (i.e., of equivalence classes of monomorphisms with codomain $B$; two monomorphisms $m:A\rightarrowtail B$ and $m':A'\rightarrowtail B$ are said to be equivalent if there exists an isomorphism $i:A\rightarrow A'$ such that $m=im'$). 

\begin{corollary}\label{well-powered}
Let $\Lambda$ be a left hereditary ring such that the injective envelope of $_{\Lambda}\Lambda$ is projective. Then the stable category of $\Lambda$ is well-powered.
\end{corollary}

\begin{proof}
Suppose $B$ is a $\Lambda$-module, and a subobject of $B$ in the stable category is represented by a monomorphism 
$\underline{f}:A \to B$. By Theorem~\ref{normal}, $\underline{f}$ is a kernel of some $\underline{g}: B \to C$. By Theorem~\ref{characterization}, $\underline{\ker g}$ is also a kernel of 
$\underline{g}$. In any category, two kernels of the same morphism define the same subobject of the domain, hence there is an isomorphism $\underline{i} : A \to \Ker g$ making the diagram 
\[
\xymatrix
	{
	A \ar[r]^{\underline{f}} \ar[d]^{\underline{i}} & B \ar@{=}[d] \\
	\Ker g \ar[r]^{\underline{\ker g}} & B
	}
\] 
commute. Thus, the class of $\underline{f}$ (as a subobject of $B$) can be represented by $\underline{\ker g}$. But the latter (as a morphism in the stable category) is represented by the monomorphism $\ker g$. It is clear that if two such monomorphisms are equivalent as subobjects of $B$ in 
$\Lambda$-$\mathrm{Mod}$, then the same is true for the corresponding subobjects in the stable category. Thus, we have embedded the class of subobjects of $B$ in the stable category in the class of subobjects of $B$ in $\Lambda$-$\mathrm{Mod}$. Since the latter is a set, the stable category is well-powered.
\end{proof}

\begin{remark}\label{colby-rutter}
 Left hereditary rings $\Lambda$ such that the injective envelope of $_{\Lambda}\Lambda$ is projective were characterized in~\cite[Theorem~3.2]{CR}. These are precisely finite direct products of \texttt{complete blocked triangular matrix rings} over division rings. Recall that $\Lambda$ is called a complete blocked triangular matrix ring over a division algebra~$D$ if there is a finite-dimensional $D$-space $V$ and a chain of subspaces
\[
V \supseteq V_{1} \supseteq \ldots \supseteq V_{k} = (0) 
\]
such that $\Lambda$ consists of all linear transformations $\lambda$ of $V$ such that $\lambda (V_{i}) \subseteq V_{i}$ for all $i = 1, \ldots , k$.
\end{remark}

\section{Normal epimorphisms in the stable category}\label{NormalEpis}

Recall that a \texttt{weak kernel} of a morphism in a pointed category (i.e., in a category with a zero object) is defined as a weak equalizer of that morphism and the zero map, which is in turn defined by removing the uniqueness requirement from the definition of equalizer.

\begin{lemma}\label{weak-kernel-of-epi}
Let $\Lambda$ be an arbitrary ring, and  
$f : A\twoheadrightarrow B$ an epimorphism of $\Lambda$-modules. Then $\underline{\ker f}$ is a weak kernel of 
$\underline{f}$. 
\end{lemma}

\begin{proof}
This is just a reformulation of the fact that the functor $\h^{X}$ is half-exact.
\end{proof}

It is well-known that if an epimorphism is normal and has a kernel, then it is a cokernel of its kernel. One can easily notice that the same is true when ``kernel'' is replaced by ``weak kernel''. Taking into account Lemma~\ref{weak-kernel-of-epi}, we now have

\begin{lemma}\label{cokernel-of-weak-kernel}
Let $\Lambda$ be an arbitrary ring, and 
$f : A\twoheadrightarrow B$ an epimorphism of $\Lambda$-modules. Then $\underline{f}$ is a normal epimorphism if and only if  it is a cokernel of  $\underline{\ker f}$. \qed
\end{lemma}

\begin{proposition}
Let $\Lambda$ be an arbitrary ring and 
$f:A\twoheadrightarrow B$ an epimorphism of $\Lambda$-modules whose kernel is contained in the torsion submodule 
$t(A)$ of $A$. Then $\underline{f}$ is a normal epimorphism if and only if $(\Ker f)^{\ast} = 0$.
\end{proposition}

\begin{proof}
We begin by observing that as $\Ker f \subset t(A)$, 
Proposition~\ref{kernel-in-torsion} shows that $\underline{f}$ is an epimorphism. With this remark, we first establish the ``if'' part. Assuming $(\Ker f)^{\ast} = 0$, we want to show that 
$\underline{f}$ is the cokernel of $\underline{\ker f}$. Suppose
we have a commutative diagram
\[
\xymatrix
	{
	\Ker f\ar[r]^>>>>>{\ker f}\ar[d]^{\alpha}
	&A\ar[r]^>>>>>{f}\ar[d]^{g} &A/\Ker f \ar@{.>}[ld]^{s}\\
	P\ar[r]^{\beta}&C
	}
\]
of solid arrows, with $P$ projective. By assumption, $\alpha = 0$, hence $g \ker f = 0$. Then we have a dotted arrow 
$s : A/\Ker f\rightarrow C$ such that $s f=g$. The uniqueness of 
$\underline{s}$ follows from the fact that $\underline{f}$ is an epimorphism.

``Only if''. Assume that $\underline{f}$ is a normal epimorphism and let $\alpha:\Ker f\rightarrow \Lambda$ be an arbitrary linear form. Lifting the composition of $\alpha$ and the injective envelope $i' : \Lambda \rightarrow I$ of $\Lambda$ over $i := \ker f$, we have a commutative diagram 
\[
\xymatrix 
	{
	\Ker f \ar[r]^{i} \ar[d]^{\alpha} & A \ar[r]^>>>>>{f} \ar[d]^{g}
	& A/\Ker f \ar@{.>}[ld]_{s}\\
	\Lambda \ar[r]^{i'}&I
	}
\]
of solid arrows. By Lemma~\ref{cokernel-of-weak-kernel}, 
$\underline{f}$ is a cokernel of $\underline{i}$. Since 
$\underline{g}\underline{i}=0$, there is a homomorphism $s : A/\Ker f \rightarrow I$ such that
$\underline{s}\underline{f}=\underline{g}$. This implies that
$sf = g + h$, for some $h:A\rightarrow I$ with $\underline{h}=0$. Then $0 = sfi = gi + hi$. But $hi = 0$, by Lemma~\ref{torsion}. Therefore $gi=0$. Since $i'$ is a monomorphism, $\alpha$ is zero.
\end{proof}

\begin{corollary}
Let $\Lambda$ be any ring and $f:A\twoheadrightarrow B$ an epimorphism of $\Lambda$-modules. If $A^{\ast} = 0$, then 
$\underline{f}$ is a normal epimorphism if and only if
$(\Ker f)^{\ast} = 0$. \qed
\end{corollary}

Combining this corollary with Lemma~\ref{hered-stable}, we immediately have 

\begin{corollary}
Let $\Lambda$ be a left hereditary ring and $f:A\twoheadrightarrow B$ an epimorphism of $\Lambda$-modules. If $A$ is stable, then 
$\underline{f}$ is a normal epimorphism if and only if $\Ker f$ is stable. \qed
\end{corollary}

Next, for a left hereditary ring $\Lambda$, we want to give a necessary and sufficient condition for an epimorphism to represent a \texttt{normal} epimorphism in the stable category. Recall from 
page~\pageref{brackets} that, for an epimorphism 
$f : A \twoheadrightarrow B$, the corresponding short exact sequence $0 \to \Ker f \overset{i}{\to} A \overset{f}{\to} B \to 0$ was denoted by $[f]$. For $\alpha : \Ker f \to C$, taking a pushout of 
$(\alpha, i)$ results in a commutative diagram 
\[
\xymatrix
	{
	0 \ar[r] 
	& \Ker f \ar[r]^{i} \ar[d]^{\alpha} 
	& A \ar[d] \ar[r]^{f} 
	& B \ar@{=}[d] \ar[r] 
	& 0
\\
	0 \ar[r] 
	& C \ar[r]  
	& D  \ar[r] 
	& B \ar[r] 
	& 0
	}
\] 
the bottom row of which will be denoted by $\alpha[f]$.

\begin{lemma}\label{splitting}
Let $\Lambda$ be an arbitrary ring, and 
$f:A\twoheadrightarrow B$ an epimorphism of 
$\Lambda$-modules. Given a homomorphism 
$\alpha : \Ker f \to X$, the following conditions are equivalent:
\begin{enumerate}
 \item $\alpha[f]$ splits;
 \smallskip
 \item any chain map
\[
\xymatrix
	{
	0 \ar[r] 
	& \Ker f \ar[r]^{i} \ar[d]^{\alpha} 
	& A \ar[d]^{g} \ar[r]^{f} 
	& B \ar[d]^{h} \ar[r] & 0
\\
	0 \ar[r] 
	& X \ar[r]  
	& Y  \ar[r] 
	& Z \ar[r] 
	& 0
	}
\]
where the bottom row is exact, is null-homotopic;
\item in any commutative diagram 
\[
\xymatrix
	{
	0 \ar[r]
	& \Ker f \ar[r]^{i} \ar[d]^{\alpha} 
	& A \ar[d]^{g} \ar[r]^{f} \ar@{.>}[ld]^{s'}
	& B \ar[d]^{h} \ar[r]  \ar@{.>}[ld]^{s}
	& 0
\\
	& X \ar[r]^{\delta} 
	& Y \ar[r]^{\gamma}
	& Z \ar[r] 
	& 0
	}
\]
of solid arrows with exact rows there exist $s : B \to Y$ and $s' : A \to X$ such that $g = sf + \delta s'$ and $h = \gamma s$.
\end{enumerate}
\end{lemma}

\begin{proof}
(1) $\Rightarrow$ (2). Suppose $0 \to X \to Y \to Z \to 0$ is exact.
Taking a pushout of ($\alpha, i$), we have a commutative diagram  
\begin{equation}\label{roof}
\begin{gathered}
\xymatrix
	{
	&\Ker f \ar[rr]^{i} \ar[dl]_{\alpha}\ar '[d]^{\alpha}[dd]
	&
	&A \ar[dl]  \ar '[d]^{g} [dd] \ar[rr]^{f} 
	&
	&B \ar@{=}[ld] 	\ar[dd]^>>>>>>>>>>>>>>>>>{h}
\\
	X \ar@{=}[dr] \ar[rr] 
	&
	& D \ar@{.>}[dr]^{b} \ar[rr] 
	&
	& B \ar@{.>}[dr]^{b'} & 
\\
	& X \ar[rr] 
	&
	& Y \ar[rr] 
	&
	& Z 
	}
\end{gathered}
\end{equation}
of solid arrows, where the rows are short exact sequences. By the universal property of pushouts, we have a dotted arrow $b$, and therefore a dotted arrow $b'$, making the two squares and the triangle incident with $b$ commute. As $f$ is an epimorphism, it is now clear that the triangle on the right also commutes. As the triangle on the left is trivially commutative, we have that the vertical chain map factors through the contractible complex $\alpha[f]$, and is therefore null-homotopic.
\smallskip

(2) $\Rightarrow (3)$. 
Suppose $X \overset{\delta}{\to} Y \to Z \to 0$ is exact. Using the same construction as in the proof of the implication (1) $\Rightarrow$ (2), we have the same  commutative diagram~\eqref{roof}; the only difference is that the map 
$\delta : X \longrightarrow Y$ in the copy is no longer assumed to be monic:
\[
\xymatrix
	{
	&\Ker f \ar[rr]^{i} \ar[dl]_{\alpha}\ar '[d]^{\alpha}[dd] 
	&
	&A \ar[dl]_{k}  \ar '[d]^{g} [dd] \ar[rr]^{f} \ar@{.>}[llld]^{s'}
	&
	&B \ar@{=}[ld] 	\ar[dd]^>>>>>>>>>>>>>>>>>{h} \ar@{.>}[llld]^{s}
\\
	X \ar@{=}[dr] \ar[rr]_>>>>>>>{i'} 
	&
	& D \ar[dr]^{b} \ar[rr] 
	&
	& B \ar[dr]^{b'} & 
\\
	& X \ar[rr]^{\delta} 
	&
	& Y \ar[rr] 
	&
	& Z 
	}
\]
The middle row, being a pushout of a short exact sequence, is still short exact. 
By the assumption, the chain map $(\alpha, k, 1_{B})$ is null-homotopic. In particular, there are maps $s : B \to D$ and $s' : A \to X$ such that $k = sf +i's'$. Composing this equality with $b$ and using the commutativity of the square
$XXDY$, we have the desired assertion for $g$. Since $f$ is epic, the equality for $h$ now follows easily.

(3) $\Rightarrow (1)$. Specialize the given diagram to the pushout diagram of 
$(\alpha, i)$. Then~$s$ gives a splitting for $\gamma$, showing that $\alpha[f]$ is split.
\end{proof}

\begin{proposition}\label{pushout-split}
Let $\Lambda$ be an arbitrary ring and $f:A\twoheadrightarrow B$ an epimorphism of $\Lambda$-modules. If $\underline{f}$ is a normal epimorphism then the equivalent conditions of Lemma~\ref{splitting} hold for any $\alpha : \Ker f \to P$ with a projective $P$.
\end{proposition}

\begin{proof}
Suppose $\Lambda$ is arbitrary and $\underline{f}$ is a normal epimorphism. Let  $\alpha : \Ker f \to P$ be an arbitrary homomorphism with a projective $P$. Taking a pushout of  
$(\alpha, i)$, where $i := \ker f$, yields a commutative diagram 
\[
\xymatrix
	{
	0 \ar[r] & \Ker f \ar[r]^{i} \ar[d]^{\alpha} 
	& A \ar[d]^{g} \ar[r]^{f} & B \ar@{.>}[ld]_>>>>>>>>>{s}\ar@{=}	[d] \ar[r] & 0\\
	0 \ar[r] & P \ar[r]  & D  \ar[r]^{\pi} & B \ar[r] & 0
	}
\]  
of solid arrows with exact rows. Since $\underline{f}$ is normal, it 
is a cokernel of its weak kernel~$\underline{i}$, and hence there is $s : B \to D$ such that $\underline{g} = \underline{sf}$. Since 
$\underline{f}$ is an epimorphism, $\underline{1_{B}} = \underline{\pi s}$. This means that $1_{B} = \pi s + h$ for some endomorphism $h$ of $B$ factoring through a projective. Since $\pi$ is an epimorphism, $h = \pi s'$ for some $s' : B \to D$. Then $s + s'$ is a splitting for $\pi$, showing that $\alpha [f]$ splits.
\end{proof}

\begin{lemma}\label{weak-cokernel}
Let $\Lambda$ be an arbitrary ring and $f:A\twoheadrightarrow B$ an epimorphism of $\Lambda$-modules. If the equivalent conditions of Lemma~\ref{splitting} hold for any $\alpha : \Ker f \to P$ with a projective $P$, then $\underline{f}$ is a weak cokernel of $\underline{\ker f}$.
\end{lemma}

\begin{proof}
Suppose we have $g : A \to X$ such that $gi = \delta \alpha$, where $i := \ker f$ and the codomain $P$ of~$\alpha$ is projective. This results in a commutative diagram
\[
\xymatrix
	{
	0 \ar[r] 
	& \Ker f \ar[r]^{i} \ar[d]^{\alpha} 
	& A \ar@{.>}[dl]_{s'} \ar[d]_{g} \ar[r]^{f} 
	& B \ar@{.>}[ld]_>>>>>>>{s} \ar[d] \ar[r] & 0
\\
	& P \ar[r]^{\delta}
	& Y \ar[r]
	&  Z \ar[r] & 0
	}
\] 
of solid arrows with exact rows. By Lemma~\ref{splitting}, there are maps $s$ and $s'$ such that $g = sf + \delta s'$, showing that $\underline{g} = \underline{sf}$.
\end{proof}

\begin{theorem}\label{normal-epi}
 Let $\Lambda$ be a {left hereditary} ring and 
 $f:A\twoheadrightarrow B$ an epimorphism of $\Lambda$-modules. The following conditions are equivalent:
 
\begin{enumerate}

\item $\underline{f}$ is a normal epimorphism;
\smallskip

\item the equivalent conditions of Lemma~\ref{splitting} hold for any  $\alpha : \Ker f \to P$ with a projective $P$.

\end{enumerate} 

\end{theorem}

\begin{proof}
(1) $\Rightarrow$ (2).  This is Proposition~\ref{pushout-split}.
\smallskip

(2) $\Rightarrow$ (1). As Lemma~\ref{weak-cokernel} shows that 
$\underline{f}$ is a weak cokernel of the class of the kernel~$i$ of
$f$, we only need to show that $\underline{f}$ is an epimorphism. For that, we shall use Theorem~\ref{char-epi}. Taking an arbitrary $h : A \to Q$ with $Q$ projective and constructing $h \ulcorner^{f}$, we have a commutative diagram 

\[
\xymatrix
	{
	0 \ar[r] 
	& \Ker f \ar[r] \ar[d]^{\alpha} 
	& A \ar[d]^{h} \ar[r]^{f} 
	& B \ar[d] \ar[r] 	
	& 0
\\
	0 \ar[r] 
	& P \ar[r]  
	& Q  \ar[r]^{h \ulcorner^{f}} 
	& D \ar[r] 
	& 0
	}
\] 
with exact rows.  Since $\Lambda$ is left hereditary, $P = \Ker (h \ulcorner^{f})$ is projective. By the assumption, the chain map $\langle h \ulcorner^{f}\rangle$ is null-homotopic. The result now follows from Theorem~\ref{char-epi}.
\end{proof}

Recall that a category is said to be \texttt{conormal} (in the sense of \cite{Mi}) if any epimorphism is normal.

\begin{lemma}\label{conormal}
Let $\Lambda$ be a left hereditary ring. If the injective envelope of
$_{\Lambda}\Lambda$ is projective, then $\Lambda$-$\Modst$ is conormal.
\end{lemma}

\begin{proof}
Let $f:A\rightarrow B$ be an epimorphism such that 
$\underline{f}$ is also an epimorphism. To prove that 
$\underline{f}$ is normal we use Theorem~\ref{normal-epi}. Thus, we need to show that $\alpha[f]$ splits for any 
$\alpha : \Ker f \to P$ with a projective $P$. Taking a pushout of 
$(\alpha, i)$ we have a commutative diagram 
\[
\xymatrix
	{
	0 \ar[r] & \Ker f \ar[r]^{i} \ar[d]^{\alpha} 
	& A \ar[d]_{h} \ar[r]^{f} 
	& B \ar@{=}[d] \ar[r] 
	& 0\\
	0 \ar[r] & P \ar[r]^{\beta}  & D  \ar[r]^{\pi} & B \ar[r] & 0
	}
\]  
with exact rows. Let $\iota : P \to I$ be the injective envelope of 
$P$. Extending $\iota$ over $\beta$, we have a map $\varphi : D \to I$. Taking a pushout of $(\varphi h, f)$, we have a commutative diagram
\[ 
\xymatrix
	{
	0 \ar[r] 
	& \Ker f \ar[r]^{i} \ar[d]^{\alpha} 
	& A \ar[d]_{h} \ar[rr]^{f} 
	&& B \ar@{=}[d] \ar[r] \ar@{.>}[llddd]_{s}
	& 0\\
	0 \ar[r] 
	& P \ar[r]^{\beta} \ar@{=}[dd] 
	& D \ar[rr]^>>>>>>>>>>>>>>>{\pi} \ar[dd]_{\varphi} 
	&& B \ar[r] \ar[dd]^{\varphi'} \ar[dl] 
	& 0\\
	  &&&Y \ar@{.>}[dr]^{u} &&\\
	0 \ar[r]
	& P \ar[r]^{\iota}
	& I \ar[rr]^>>>>>>{\pi'} \ar[ur]_{(\varphi h) \ulcorner^{f}}
	&& X \ar[r]
	& 0
	}
\]  
of solid arrows, where $X : = \mathrm{Coker}\,\iota$. The two triangles incident with the dotted arrow $u : Y \to X$, which arises by the universal property of pushouts, also commute. As we remarked before, $I$ is projective by \cite{CR}. By Theorem~\ref{char-epi}, there is a dotted arrow $s : B \to I$ making the square $BIYB$ commute.
A simple diagram chase now shows that $\pi' s = \varphi'$. This means that the middle row, which is a pullback of the bottom row, is split. Hence $\underline{f}$ is a normal epimorphism, showing that the stable category of $\Lambda$ is conormal.
\end{proof}

Now we are going to deal with the issue whether the converse of the above theorem is true. First, we give

\begin{lemma}\label{proj-embed}
Let $\Lambda$ be a left hereditary ring, $A$ a stable $\Lambda$-module, $P$ a nonzero projective submodule of $A$, and $p : A \to A/P$ the canonical projection. Then~$\underline{p}$ is a bimorphism (i.e., both an epimorphism and a monomorphism) in the stable category, but not a split monomorphism, and hence not an isomorphism.
\end{lemma}

\begin{proof}
By Theorem~\ref{characterization}, $\underline{p}$ is a monomorphism. Corollary~\ref{epi-cases}~(a) shows that 
$\underline{p}$ is an epimorphism. By Lemma~\ref{split-mono} and Lemma~\ref{hered-stable}, $\underline{p}$ is not a split monomorphism.
\end{proof}

We shall now describe a large class of left hereditary rings for which a pair $(P,A)$ satisfying the conditions of 
Lemma~\ref{proj-embed} exists.

\begin{theorem}\label{existence}
 Suppose that $\Lambda$ is a left hereditary ring such that:
\begin{itemize}
 \item[(a)] $\Lambda$ has the DCC on direct summands of 
 $_{\Lambda}\Lambda$, and
 \smallskip
 \item[(b)] the injective envelope of $_{\Lambda}\Lambda$ is not projective.
\end{itemize}
Then there is a nonzero projective $\Lambda$-module with a stable injective envelope.
\end{theorem}

\begin{proof}
 Let $P_{0} := \Lambda$ and let $i_{0} : P_{0} \to I_{0}$ be an injective envelope. If $I_{0}$ is stable, we are done. If not, then 
 $I_{0} = I_{1} \oplus Q_{0}$, where $Q_{0}$ is a nonzero projective injective and, by (b), $I_{1}$ is not projective, and hence nonzero. Let $f_{0}$ be the composition $P_{0} \overset{i_{0}}{\to} I_{1} \oplus Q_{0} \to Q_{0}$, where the last map is the canonical projection, and let $P_{0} \overset{e}{\to} f_{0}(P_{0}) \overset{m}{\to} Q_{0}$ be the mono-epi factorization of $f_{0}$. Since $\Lambda$ is left hereditary and $Q_{0}$ is projective, $f_{0}(P_{0})$ is projective too, and we have a commutative diagram 

\[
\xymatrix
	{
	&0 \ar[d] 
	&0\ar[d] 
	&0 \ar[d] 
	& 
\\
	0 \ar[r] 
	& P_{1} \ar[d] \ar[r]^{i_{1}} 
	& I_{1} \ar[d] \ar[r] 
	& I_{1}/P_{1} \ar[d] \ar[r] 
	& 0
\\
	0 \ar[r] 
	& P_{0} \ar[d]^{e} \ar[r]^{i_{0}} 
	& I_{1} \oplus Q_{0} \ar[d] \ar[r] 
	& I_{0}/P_{0} \ar[d] \ar[r] 
	& 0 
\\
	0 \ar[r] 
	& f_{0}(P_{0}) \ar[d] \ar[r]^{m} 
	& Q_{0} \ar[d] \ar[r] 
	& Q_{0}/f_{0}(P_{0}) \ar[d] \ar[r] 
	& 0 
\\
	&0  
	&0 
	&0 
	&
	}
\] 
with exact rows and columns. Now we list some of the  properties of $P_{1}$ and $i_{1}$.

\begin{itemize}
\item Since $f_{0}(P_{0})$ is projective, the first column is split exact, i.e., $P_{1}$ is a direct summand of $P_{0}$.
 \item $i_{1} : P_{1} \to I_{1}$ is an injective envelope.
 Indeed, a simple diagram chase shows that since $i_{0}$ is essential, so is $i_{1}$.
 \item Since $I_{1}$ is a nonzero module, the previous observation shows that so is $P_{1}$. 
 \item $P_{1}$ is a proper direct summand of $P_{0}$. If not, then the inclusion $P_{1} \to P_{0}$ becomes an equality, making both $i_{0}$ and $i_{1}$ injective envelopes of the same module, contradicting the maximality of the essential extension $i_{1}$ (since $Q_{0}$ is nonzero). 
\end{itemize} 
Thus, $P_{1}$ is a nonzero (projective) summand of $P_{0}$ whose injective envelope $I_{1}$ is not projective, and we can iterate the above construction. As a result, we get a strictly descending chain of direct summands $P_{0} \supset P_{1}  \supset \ldots $.
When it stabilizes, we have a short exact sequence 
$0 \to P_{n} \to I_{n} \to I_{n}/P_{n} \to 0$, where $P_{n}$ is a nonzero projective and $I_{n}$ is a stable injective.
\end{proof}

\begin{remark}
 It is not difficult to see that $\Lambda$ has the DCC on direct summands of $_{\Lambda}\Lambda$ if and only if it has the ACC on direct summands of $_{\Lambda}\Lambda$. Thus this class of rings includes left noetherian rings.
\end{remark}

\begin{proposition}
Let $\Lambda$ be a left hereditary ring with a conormal stable category. Then either the injective envelope of $_{\Lambda}\Lambda$ is projective or $\Lambda$ does not satisfy the DCC on direct summands of $_{\Lambda}\Lambda$.
\end{proposition}

\begin{proof}
Assume that the injective envelope of $\Lambda$ is not projective
and $\Lambda$ satisfies the DCC on the direct summands of $\Lambda$. By Theorem~\ref{existence}, there is a nonzero projective module $P$ with a stable injective envelope $I$. Let $p$ be the canonical projection $I \rightarrow I/P$. By Lemma~\ref{proj-embed}, $\underline{p}$ is both a monomorphism and an epimorphism. By assumption, $\underline{p}$ is a normal epimorphism. But being a monomorphism and a normal epimorphism, $\underline{p}$ must be an isomorphism, which contradicts Lemma~\ref{proj-embed}.
\end{proof}

In regard to the main result of this paper (Theorem~\ref{main} below), recall the well-known fact that any morphism in an abelian category factors as an epimorphism followed by a monomorphism, and such factorizations are functorial in a certain sense. A  conceptual context to discuss this phenomenon is provided by the notion of \texttt{factorization system} in a category, as introduced by Freyd-Kelly~\cite{FK}. It is defined as a pair of morphism classes $(\mathbb{E},\mathbb{M})$ such that $\mathbb{E}$ and $\mathbb{M}$ contain all isomorphisms, are closed under composition with them, and satisfy the following conditions:

\begin{enumerate}
 \item every morphism $\alpha$ admits an
$(\mathbb{E},\mathbb{M})$-factorization, i.e., there are morphisms
$e\in \mathbb{E}$ and $m\in \mathbb{M}$ with $\alpha=me$.

\item\label{arrow-down} for each $e \in \mathbb{E}$, $m\in \mathbb{M}$, and a commutative square
\[
\xymatrix
	{
	A \ar[r]^{e} \ar[d]_{\alpha} 
	& B \ar[d]^{\beta} \ar@{.>} [ld]_{\delta}
	\\ C \ar[r]^{m} 
	&D
	}
\]
of solid arrows, there exists a unique $\delta : B \to C$ with $\alpha=\delta e$ and $\beta=m\delta$.\footnote{For more on factorization systems, see \cite{AHS}, \cite{CHK}, and \cite{Z}.}
\end{enumerate}

\smallskip

In this language, the morphism pair $(Epi, Mono)$ is a factorization system in any abelian category (here, $Epi$ and $Mono$ denote the classes of epimorphisms and, respectively, monomorphisms). However, there are non-abelian categories where the pair $(Epi, Mono)$ forms a factorization system, too. 

Condition~(\ref{arrow-down}) above shows that the intersection of  $\mathbb{E}$ and $\mathbb{M}$ is contained in (and hence coincides with) the class of isomorphisms for any factorization system $(\mathbb{E,M})$. Combining this with Lemma~\ref{proj-embed} and Theorem~\ref{existence}, we now have

\begin{proposition}
Let $\Lambda$ be a left hereditary ring, and suppose the pair of morphism classes $(Epi,Mono)$ be a factorization system on the stable category. Then either the injective envelope of $_{\Lambda}\Lambda$ is projective or $\Lambda$ does not satisfy the DCC on 
direct summands of $_{\Lambda}\Lambda$. \qed
\end{proposition}

\section{Main theorem}\label{Cokernels}

Recall that a set of objects $S$ of a category $\mathbf{C}$ is said to be \texttt{cogenerating} if, for any pair of distinct morphisms $h,h':A\rightarrow B$, there is an object $X$ from $S$ and a morphism $s:B\rightarrow X$ such that $sh\neq sh'$. As is well-known, if $\mathbf{C}$ has products, then the equivalent condition is that for any object $C$ there are a set $I$  and a monomorphism $C\rightarrowtail \prod_{i\in I}X_{i}$, where the $X_{i}$ are in $S$. The next result is common knowledge, but for the convenience of the reader, we provide a proof.

\begin{lemma}\label{cocomplete}
Any complete well-powered category $\mathbf{C}$ with a cogenerating set is cocomplete.
\end{lemma}

\begin{proof}
 For any small category $\mathbf{X}$, consider the diagonal functor $\Delta: \mathbf{C}\rightarrow \mathbf{C}^{\mathbf{X}}$, where~$\mathbf{C}^{\mathbf{X}}$ is the functor category. Obviously, $\Delta$ preserves limits. Thus, the special adjoint functor theorem implies that $\Delta$ has a left adjoint $L$. It is clear that $L(F)$ is the colimit of $F$ for any functor 
$F : \mathbf{X} \rightarrow \mathbf{C}$.
\end{proof}

\begin{proposition}\label{lp-rc}
Any left hereditary ring $\Lambda$ whose injective envelope is
projective, is left perfect and right coherent.
\end{proposition}

\begin{proof}
By Remark~\ref{colby-rutter}, any such ring is a finite direct product of complete blocked triangular rings over a division ring. This implies that it is left and right artinian, and hence left and right perfect. Since any right artinian ring is right noetherian, it is right coherent, too.
\end{proof}

\begin{proposition}\label{cogenerating}
Let $\Lambda$ be a left hereditary, left perfect, and right coherent ring. Then $\Lambda$-$\Modst$ has a cogenerating set.
\end{proposition}

\begin{proof}
By Corollary~\ref{hered-complete},  $\Lambda$-$\Modst$ is complete. Now it suffices to apply the well-known fact that the category of modules has a cogenerating object, together with Proposition~\ref{products} and Theorem~\ref{characterization}.
\end{proof}

Lemma~\ref{cocomplete}, Corollary~\ref{hered-complete}, Corollary~\ref{well-powered}, Proposition~\ref{lp-rc}, and Proposition~\ref{cogenerating} imply

\begin{proposition}\label{stable-cocomplete}
Let $\Lambda$ be a left hereditary ring, and suppose the injective envelope of $\Lambda$ is projective. Then $\Lambda$-$\Modst$ is cocomplete. \qed
\end{proposition}

Now we are ready to prove the main result of this paper.

\begin{theorem}\label{main}
Let $\Lambda$ be a left hereditary ring. The stable category  
$\Lambda$-$\Modst$ of left $\Lambda$-modules is abelian if and only if the injective envelope of $_{\Lambda}\Lambda$ is projective. The class of such rings consists precisely of finite direct products of complete blocked triangular matrix algebras over division rings.
\end{theorem}

\begin{proof}
Recall that a category is abelian if and only if it has a zero
object, is finitely complete and finitely cocomplete, and, moreover, is both normal and conormal. Now apply Theorem~\ref{normal}, Lemma~\ref{conormal}, Proposition~\ref{kernel}, and Proposition~\ref{stable-cocomplete}. The last assertion of the theorem follows from Remark~\ref{colby-rutter}.
\end{proof}

\end{document}